\documentclass[11pt]{article}
\usepackage{amsthm}
\usepackage{amssymb}
\usepackage{amsmath,enumerate}
\usepackage{comment}
\usepackage{thm-restate}
\usepackage{url}
\usepackage{hyperref}
\usepackage{graphicx}
\usepackage{subfigure}
\usepackage[noabbrev,capitalise]{cleveref}
\usepackage[affil-it]{authblk}
\usepackage{color}
\usepackage[normalem]{ulem}

\usepackage[margin=1in]{geometry}

\theoremstyle{plain}
\newtheorem{thm}{Theorem}[section]
\newtheorem{lem}[thm]{Lemma}

\newtheorem{cor}[thm]{Corollary}

\newtheorem{conj}[thm]{Conjecture}

{\noindent \emph{Proof.} {}{#1}{}}{\hfill
	$\Diamond$\vspace{1em}}

\theoremstyle{plain} 
\newcommand{\thistheoremname}{}
\newtheorem{genericthm}[section]{\thistheoremname}

\theoremstyle{definition}
\newtheorem{definition}[thm]{Definition}
\newtheorem{example}[thm]{Example}
\newtheorem{cla}[thm]{Claim}

\title{Odd clique minors and chromatic bounds of \{3$K_1$, paraglider\}-free graphs}
\author{Yuqing Ji$^a$, Yue Wang$^{a}$, Yujun Yang$^b$, Xia Zhang$^{a,}$\thanks{Corresponding author.}\\
 \footnotesize{ $^a$School of Mathematics and Statistics, Shandong Normal University, Jinan, 250358, China\\
 $^b$ School of Mathematics and Information Science, Yantai University, Yantai, Shandong, 264005, China
\bigskip

  Emails:  yqji98@aliyun.com (Y. Ji), wangyue$\_$math@163.com (Y. Wang), yangyj@ytu.edu.cn (Y. Yang), xiazhang@sdnu.edu.cn (X. Zhang)}}
\date{ }
\begin{document}
\maketitle

\begin{abstract}
A paraglider, house, 4-wheel, is the graph that consists of a cycle $C_4$ plus an additional vertex adjacent to three vertices, two adjacent vertices, all the vertices of the $C_4$, respectively. For a graph $G$, let $\chi(G)$, $\omega(G)$ denote the chromatic number, the clique number of $G$, respectively. Gerards and Seymour from 1995 conjectured that every graph $G$ has an odd $K_{\chi(G)}$ minor. In this paper, based on the description of graph structure, it is shown that every graph $G$ with independence number two satisfies the conjecture if one of the following is true: $\chi(G) \leq 2\omega(G)$ when $n $ is even, $\chi(G) \leq 9\omega(G)/5$ when $n$ is odd, $G$ is a quasi-line graph, $G$ is $H$-free for some induced subgraph $H$ of paraglider, house or $W_4$. Moreover, we derive an optimal linear $\chi$-binding function for \{3$K_1$, paraglider\}-free graph $G$ that $\chi(G)\leq   \max\{\omega(G)+3, 2\omega(G)-2\}$, which improves the previous result, $\chi(G)\leq 2\omega(G)$, due to Choudum, Karthick and Shalu in 2008.

\end{abstract}

\section{Introduction}

All graphs considered in this paper are simple and finite. Let $G$ be a graph with vertex set $V (G)= \{v_{1}, v_{2}, \cdots, v_{n}\}$ and edge set $E(G)$. Denote the order of $G$ by $|V(G)| = |G|$ and the $complement$ of $G$ by $\overline{G}$. For $S \subseteq V(G)$, let $G[S]$ be the subgraph of $G$ induced by $S$. For convenience, we write $H \sqsubseteq G$ if $H$ is an induced subgraph of $G$. We use $\chi(G)$, $\alpha(G)$, $\omega(G)$ to denote the chromatic number, the independence number, the clique number of $G$, respectively. We say $G$ is $t$-chromatic if $\chi(G) = t $. A graph $G$ is called $perfect$ if $ \chi(H) = \omega(H)$ for every induced subgraph $H$ of $G$, otherwise it is called imperfect. Gy\'arf\'as \cite{gya} extended the concept of perfection
and called a hereditary class $\mathcal{G}$ of graphs $\chi$-$bound$ if there is a function $f$ such that $\chi(G) \leq f(\omega(G))$ for every $G \in \mathcal{G}$. The function $f$ is called a $\chi$-$binding$ $function$ of $\mathcal{G}$ and is called $optimal$ if  no smaller function is a $\chi$-binding function for $\mathcal{G}$. In particular, for every $G \in \mathcal{G}$, we say that $G$ is $linearly$ $\chi$-$bound$ if $\chi(G) \leq c\omega(G)$ for some constant $c$. If $\mathcal{G}$ is the class of $H$-free graphs for some graph $H$, then $f$ is denoted by $f_H$.

For $v \in V(G)$, let $N_{G}(v)$ denote the set of all neighbors of $v$ in $G$. For a vertex set $S \subseteq V(G)$, $N_{G}(S) = \cup_{v \in S}  N_{G}(v)\setminus S$. As usual, we denote the cycle, the path and the complete graph on $n $ vertices by $C_{n}, P_{n}$ and $K_{n}$, respectively. For $k \geq 4$, a $k$-$wheel$ $W_k$ is the graph consisting of a cycle $C_k$ plus an additional vertex adjacent to all the vertices of the $C_k$. A $house$ is the graph that consists of a cycle $C_4$ plus an additional vertex adjacent to two adjacent vertices of the $C_4$. If vertices $x$ and $y$ are connected in $G$, the  $distance$ between $x$ and $y$ in $G$, denoted by $d_{G}(x, y)$, is the length of a shortest $(x, y)$-path in $G$. For any positive integer $t$, we write $[t]$ for the set $\{1, \cdots, t\}$.

For two vertex-disjoint graphs $G$ and $H$, graph $G \cup H $ is the $union$ of $G$ and $H$ with vertex set $V(G)\cup V(H)$ and edge set $E(G)\cup E(H)$; $G + H $ is the $join$ of $G$ and $H$ with vertex set $V(G)\cup V(H)$ and edge set $\{xy| x \in V(G), y\in V(H) \} \cup E(G) \cup E(H)$. Given two disjoint non-empty subsets $S, T \subseteq V(G)$, we denote the edge set $\{uv \in E(G): u \in S$ and $v \in T \}$ by $[S, T]$. And we say that $S$ is {\it complete} to $T$, denoted by $S \thicksim T$, if every vertex in $S$ is adjacent to all vertices in $T$; $S$ is {\it anti-complete} to $T$, denoted by $S \nsim T$, if $[S, T] = \emptyset$. When graph $G_1$ is isomorphic to $G_2$, we write $G_1 \cong G_2$. For any integer $k$, use $kG$ to denote the union of $k$ copies of $G$. A graph $G$ is $H$-free if $G$ has no induced subgraph isomorphic to $H$. For a family $\mathfrak{F}$ of graphs, $G$ is $\mathfrak{F}$-free if $G$ is $F$-free for every $F \in \mathfrak{F}$. We say that graph $G$ contains graph $H$ as a {\it minor} if $H$ can be obtained from a subgraph of $G$ by contracting edges.

In 1943, Hadwiger posed the following conjecture.
\begin{conj}[Hadwiger \cite{had}]\label{conj-hadwiger}
For all $t \geq 2$, every graph with no $K_{t}$ minor is $(t-1)$-colorable.
\end{conj}

It is easy to check that Hadwiger's conjecture trivially holds for $ t \leq 3$. For $t = 4$, as shown by Hadwiger himself \cite{had} and Dirac \cite{dir}, respectively, it is also true. Wagner \cite{wag} proved that the case $t = 5$ of Hadwiger's conjecture is, in fact, equivalent to the Four Color Theorem, and the same was shown for $t = 6$ by Robertson, Seymour and Thomas \cite{rst} in 1993. It remains open for $t \geq 7$. For more background on Hadwiger's conjecture, we refer to Seymour's survey \cite{sem} in 2016 and a very recent survey by Norin \cite{n}.

One strengthening of Hadwiger's conjecture is to consider the odd-minor variant. We call that $G$ has an {\it odd clique minor} of size at least $t$ if there are $t$ vertex disjoint trees in $G$ such that every two of them are joined by an edge, in addition, all the vertices of trees can be two-colored in such a way that the edges within the trees are bi-chromatic, but the edges between trees are monochromatic (and hence the vertices of all one-vertex trees must receive the same color). Obviously, any graph that has an odd $K_{t}$ minor certainly contains $K_{t}$ as a minor.

In 1995, Gerards and Seymour (see Section 6.5 in \cite{jto}) posed the following conjecture.

\begin{conj} [Gerards and Seymour \cite{jto}]\label{conj-odd-Hdwiger}
For all $t \geq 2$, every graph with no odd $K_{t}$ minor is $(t-1)$-colorable.
\end{conj}

This conjecture is referred to as ``Odd Hadwiger's Conjecture''. It is far stronger than Hadwiger's conjecture. The case $t \leq 4$ was proved by Catlin \cite{cat} in 1978. Guenin \cite{gue} announced a solution of the case $t = 5$ at a meeting in Oberwolfach in 2005. It remains open for $t \geq 6$. For graphs with no odd $K_{t}$-minor, the current best known upper bound of chromatic number is $O(t\log \log t)$ due to Steiner \cite{s}.

In a $\chi(G)$-coloring of $G$, each color-class has size at most $\alpha(G)$. Thus, Hadwiger's conjecture implies the following weaker version.

\begin{conj}[Duchet and Meyniel \cite{dm}]\label{conj-D-M}
Any graph $G$ of order $n$  has a $K_{\lceil n/\alpha(G) \rceil}$  minor.
\end{conj}

Let $oh(G)$ denote the largest integer $\ell$ such that $G$ has an odd $K_\ell$ minor. Odd Hadwiger's Conjecture states that $oh(G)\geq \chi(G)$ for every graph $G$. In \cite{JSWZ}, the first author, Song, Weiss and the fourth author posed an odd-minor variant of Duchet-Meyniel Conjecture.

\begin{conj}[Ji, Song, Weiss, and Zhang \cite{JSWZ}]\label{conj-ODM}
For any graph $G$ of order $n$, $oh(G) \geq \lceil n/\alpha(G) \rceil$.
\end{conj}

In 2007, Kawarabayashi and Song \cite{ks} proved that any graph $G$ on $n$ vertices has an odd $K_{\lceil n/ (2\alpha(G)-1) \rceil}$ minor.
Given a graph $H$, we say that $G$ is an  $inflation$ of $H$ if $G$ can be obtained from $H$ in such a way that each vertex $v$ of $H$ is replaced by a complete graph $K^{v}$ and every vertex of $K^{v}$ is joined to all vertices of $K^{u}$ in $G$ if $uv \in E(H)$. We denote such a graph $G$ by $\mathbb{K}[H]$.
Song and Thomas \cite{st} in 2017 showed that $oh(G)\geq \chi(G)$ if graph $G$ with $\alpha(G) \geq 2$ is $\{C_4, C_5,  \dots,  C_{2\alpha(G)}\}$-free, or $G$ is an inflation of an odd cycle.

One particular interesting case of Hadwiger's conjecture is when graphs have independence number two. It has attracted more attention (see Section 4 in Seymour's survery \cite{sem} for more information). As stated in his survey, Seymour believes that if Hadwiger's conjecture is true for graphs $G$ with $\alpha(G) = 2$, then it is probably true in general. Some scholars have verified Hadwiger's conjecture for $H$-free graphs $G$ on $n$ vertices with $\alpha(G) = 2$, when $H$ is any 4-vertices graph by Plummer, Stiebitz and Toft \cite{pst}, any 5-vertices graph by Kriesell \cite{kri}, 5-wheel on six vertices by Bosse \cite{bo}, any of the 33 graphs on seven, eight, nine vertices, or $K_8$ by Carter \cite{car} (computer-assisted).

In \cite{JSWZ}, it was shown that Odd Hadwiger's Conjecture and Conjecture \ref{conj-ODM} are equivalent for graphs with independence number at most two.
\begin{thm}[\cite{JSWZ}]\label{thm-ohc-alpha=2}
Let $G$ be a graph on $n$ vertices with $\alpha(G) \leq 2$. Then
\begin{center}
$oh(G)\geq\chi(G)$ if and only if $oh(G)\geq \lceil n/2\rceil$.
\end{center}
\end{thm}
In the same paper \cite{JSWZ}, the following two results had been obtained.

\begin{thm}[\cite{JSWZ}]\label{thm-ohlow}
Let $G$ be a graph on $n$ vertices with $\alpha(G) \leq 2$. If

\begin{center}
$\omega(G) \geq \left\{
                 \begin{array}{ll}
                   n/4, & \hbox{when n is even,} \\
                   (n+3)/4, & \hbox{when n is odd,}
                 \end{array}
               \right.$
\end{center}
then $oh(G) \geq \chi(G)$.

\end{thm}

\begin{thm}[\cite{JSWZ}]
Let $G$ be a graph on $n$ vertices with $\alpha(G) \leq 2$. For any induced subgraph $H'$ of $H$, if graph $G$ is $H'$-free, then $oh(G)\geq \chi(G)$, where
$$H \in \{ K_{1} + P_4, K_{2} + (K_1 \cup K_{3}), K_{1} + (K_1 \cup K_{4}), K_7^-, K_7, kite\}.$$
\end{thm}

When a $3K_1$-free graph $G$ has chromatic number upper-bounded by $c\omega(G)$ for some constant $c$, we can show that the class of such graphs $G$ satisfy Odd Hadwiger's Conjecture. First, we obtain the following result.

\begin{thm}\label{thm-oh-low2}
Let $G$ be a graph with $\alpha(G) \leq 2$. If
\begin{center}
$\chi(G) \leq \left\{
                 \begin{array}{ll}
                   2\omega(G), & \hbox{when n is even,} \\
                   9\omega(G)/5, & \hbox{when n is odd,}
                 \end{array}
               \right.$
\end{center}
then $oh(G) \geq \chi(G)$.
\end{thm}
\begin{proof}
Suppose $oh(G) < \chi(G)$. By $\alpha(G) \leq 2$ and
\begin{center}
$\chi(G) \leq \left\{
                 \begin{array}{ll}
                   2\omega(G), & \hbox{when $n$ is even,} \\
                   9\omega(G)/5, & \hbox{when $n$ is odd,}
                 \end{array}
               \right.$
\end{center}
we have that $\omega(G) \geq  n/4$ when $n$ is even; $\omega(G) \geq 5n/18$ when $n$ is odd. By the assumption and Lemma \ref{n26}, we know that $n \geq 27$ when $n$ is odd, which means that $\omega(G) \geq 5n/18 \geq (n+3)/4$ when $n$ is odd. By Theorem \ref{thm-ohlow}, $oh(G) \geq \chi(G)$, a contradiction.
\end{proof}

For any positive integers $k$ and $\ell$,  $Ramsey\ number$ $r= R(k, \ell)$ is the smallest integer such that all graphs on $r$ vertices contain either an independent set of $k$ vertices or a clique of $\ell$ vertices.

By the Kim's famous result \cite{kim} that the Ramsey number $R(3, \omega)$ has order of magnitude $\omega^{2}/ \log \omega$, $f_{3K_{1}} = \Theta(\omega^{2}/ \log \omega)$. Chudnovsky and Seymour \cite{cse} showed that if $G$ is a connected $K_{1,3}$-free graph with $\alpha(G) \geq 3$, then $\chi(G) \leq 2\omega(G)$ and this bound is asymptotically optimal. As noted in \cite{cks}, the problem of finding an optimal $f_{3K_{1}}$ is open. In Table 1, the known best chromatic bounds are shown for a $\{3K_1, H\}$-free graph $G$, where $H$ is some $3K_1$-free graph. We refer to \cite{ss} for an extensive survey on $\chi$-boundedness.

\begin{center}
\resizebox{0.6\columnwidth}{!}{
\begin{tabular}{c|c|c|c}
\hline
  $H$ & chromatic upper bound for $G$ &&\\ \hline
  &&&\\
   $2K_2$     & non-linear & $\divideontimes$ &  \cite{brsv}\\
   &&&\\
  $W_4$ / paraglider & $2\omega(G)$ &$\divideontimes$& \cite{cks} \\
  &&&\\
   $K_5$ & $f_{3K_1} \leq 8$ &$\divideontimes$&  \cite{gya}\\
   &&&\\
   $C_5$ & $\omega(G)^{\frac{3}{2}}$ & $\divideontimes$& \cite{hoa}\\
   &&&\\
  $(K_1 \cup K_2) + K_2$ / $K_5 - e$ & $\omega(G)+1$ &  $\checkmark$&    \cite{dhu, kie}\\
  &&&\\
  $C_4$ & $\lceil \frac{5\omega(G)}{4}\rceil$ &$\checkmark$& \cite{cs} \\
  &&&\\
  $K_1 + P_4$ & $\lceil \frac{5\omega(G)}{4} \rceil$ &$\checkmark$ & \cite{cks}\\
  &&&\\
  $(K_1 \cup K_3) + K_1$ / house / kite  & $\lfloor \frac{3\omega(G)}{2}\rfloor$  &  $\checkmark$& \cite{cks}\\
  &&&\\
  $K_1 \cup K_4$ & $ \lfloor  \frac{3\omega(G)}{2} \rfloor$ &$\checkmark$ & \cite{hlr}\\
  &&&\\
  paraglider & $ \max\{\omega(G)+3, 2\omega(G)-2\}$ &$\checkmark$& [this paper] \\
    &&&\\
  \hline
\end{tabular}}\\
 \small\begin{quote}
Table 1. The symbol `/' denotes `or', the symbol `$\checkmark$' indicates that the bounds are optimal, and the symbol `$\divideontimes$' indicates that the problem of finding an optimal bound is open.
 \end{quote}
\end{center}

\begin{figure}[ht]
\begin{center}
        \includegraphics[scale=0.2]{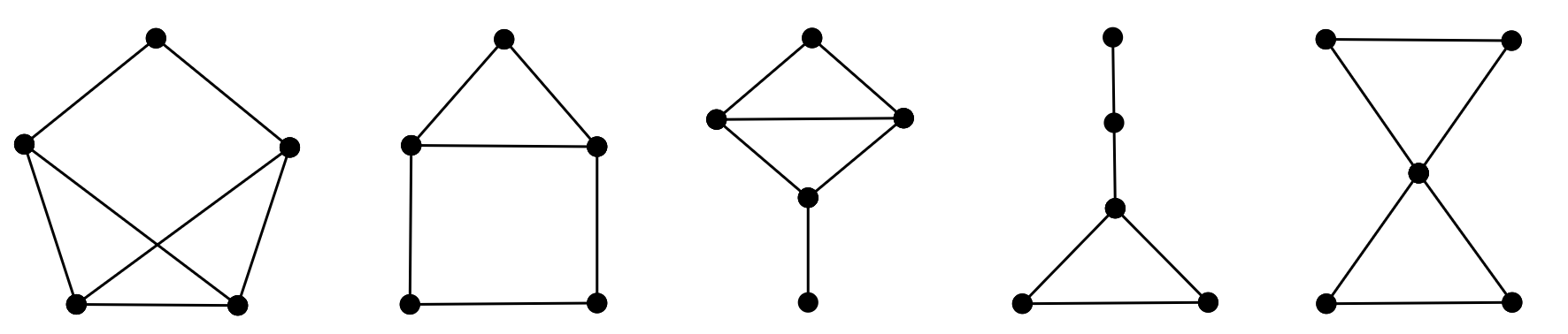}

         \caption{paraglider, house, kite, hammer and butterfly.}
         \label{fig1}
\end{center}
\end{figure}

A $paraglider$, denoted by $\overline{P_2 \cup P_3}$, is the graph that consists of a cycle $C_4$ plus an additional vertex adjacent to three vertices of the $C_4$. In this paper, we first focus on the description on the structure of $\{3K_1, \overline{P_2 \cup P_3}\}$-free graphs (see Lemmas \ref{p2-p3 1}$-\ref{p2-p3 5})$. Next we obtain a tight lower bound of clique number for the class of such graphs (Theorem \ref{thm-p2p31}). Then we confirm Odd Hadwiger's Conjecture for them (Theorem \ref{thm-t1}). In particular, we derive an optimal linear $\chi$-binding function for such class of graphs (Theorem \ref{thm-paraglider-boundedness}).

\begin{thm}\label{thm-p2p31}
Let $G$ be a $\{3K_1, \overline{P_2 \cup P_3}\}$-free graph on $n$ vertices. Then $\omega(G) \geq (n-1)/3$, and $\omega(G) = (n-1)/3$ if and only if $G \cong \overline{H^{*}}$ $($see graph $H^{*}$ in Figure $2)$.
\end{thm}

\begin{thm}\label{thm-t1}
Let $G$ be a $\{3K_1, \overline{P_2 \cup P_3}\}$-free graph. Then $oh(G) \geq  \chi(G)$.
\end{thm}
\begin{proof}
When $n \geq 13$, there is $(n-1)/3 \geq (n+3)/4$. Therefore, the conclusion follows immediately from Theorem \ref{thm-ohlow}, Theorem \ref{thm-p2p31} and Lemma \ref{n26}.
\end{proof}

\begin{thm}\label{thm-paraglider-boundedness}
Let $G$ be a $\{3K_1, \overline{P_2 \cup P_3}\}$-free graph. Then $\chi(G) \leq \max\{\omega(G)+3, 2\omega(G)-2\}$. More precisely, when $G$ is an imperfect graph, we have
\begin{center}
$\chi(G)  \leq  \left\{
                 \begin{array}{ll}
                    2\omega(G)-2, & \hbox{if $H \in \mathcal{B}_1$,}\\
                    3\omega(G)/2 -1, & \hbox{if $H \in \mathcal{B}_2$,}\\
                   \omega(G) +3, & \hbox{otherwise,}\\
                 \end{array}
               \right.$
\end{center}
where $H$ is the imperfect component of $\overline{G}$ and $\mathcal{B}_1$, $\mathcal{B}_2$ are defined in Definitions \ref{B 1}, \ref{B 2}, respectively. In particular, the bounds $2\omega(G)-2$, $3\omega(G)/2 -1$ and $\omega(G) +3$ are tight, respectively.

\end{thm}

\begin{figure}[ht]
\begin{center}
        \includegraphics[scale=0.2]{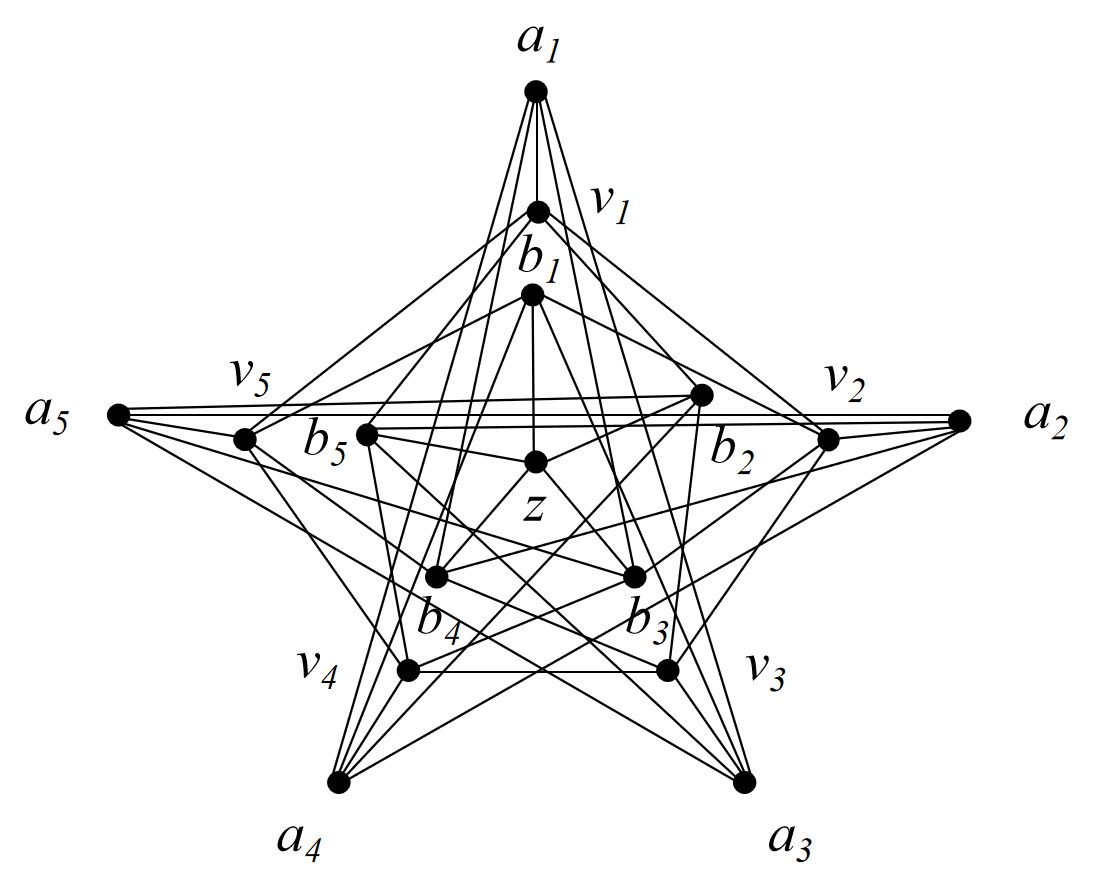}

\caption{$H^{*}.$}
\end{center}
\end{figure}

{\bf Remark.}
For an imperfect $\{3K_1, paraglider\}$-free graph $G$, by Theorem \ref{thm-p1}, $C_5 \sqsubseteq \overline{G}$. Noting that $\overline{G}$ is $P_2 \cup P_3$-free, if $\overline{G}$ is disconnected, then the component containing $C_5$ is imperfect, and the others form an independent set.

\section{Preliminaries}
\par
The following are two well-known results on perfect graphs.

\begin{thm}[\cite{lo}]\label{thm-p2}$(${\bf The Perfect Graph Theorem}$)$
A graph $G$ is perfect if and only if its complement is perfect.
\end{thm}
\begin{thm}[\cite{csrt}]\label{thm-p1}$(${\bf The Strong Perfect Graph Theorem}$)$
A graph  $G$ is perfect if and only if it does not contain $C_{2k+1}$ and $\overline{C_{2k+1}}$ as induced subgraphs for each integer $k \geq 2$.
\end{thm}

Lemma \ref{n26} is crucial to the proofs of Theorem \ref{thm-oh-low2}, Theorem \ref{thm-t1} and Corollary \ref{qhw}.
\begin{lem}\label{n26}
Let $G$ be a graph with $\alpha(G) \leq 2$. If $\omega(G) \leq 7$ or $n \leq 32$ when $n$ is even; $\omega(G) \leq 6$ or $n \leq 25$ when $n$ is odd,
then $oh(G) \geq \chi(G)$.
\end{lem}
\begin{proof}
Suppose $oh(G) < \chi(G)$. By Theorem \ref{thm-ohlow}, we see that

\begin{center}
$\omega(G) \leq \left\{
                 \begin{array}{ll}
                    \lceil n/4\rceil -1\leq (n-2)/4, & \hbox{when $n$ is even,}\\
                   \lceil (n+3)/4\rceil -1 \leq (n+1)/4, & \hbox{when $n$ is odd.}
                 \end{array}
               \right.$
\end{center}
Then
\begin{center}
$n \geq \left\{
                 \begin{array}{ll}
                    4 \omega(G) +2, & \hbox{when $n$ is even,}\\
                   4 \omega(G) -1, & \hbox{when $n$ is odd.}
                 \end{array}
               \right.$
\end{center}
Note that $R(3, \omega(G) +1) \leq 4\omega(G)$ when $\omega(G) \leq 7$; $R(3, \omega(G) +1) \leq 4\omega(G)-1$ when $\omega(G) \leq 6$. It means that if $n $ is even and $\omega(G) \leq 7$, or $n $ is odd and $\omega(G) \leq 6$, then $G$ contains a clique of order $\omega(G) +1$, a contradiction. Thus $8 \leq \omega(G) \leq (n-2)/4$ when $n$ is even; $7 \leq \omega(G) \leq (n+1)/4$ when $n$ is odd. Then $\omega(G) \geq 8$ and $n \geq 34$ when $n$ is even; $\omega(G) \geq 7$ and $n \geq 27$ when $n$ is odd, a contradiction.
\end{proof}

\section{The structure of $\{ K_3, P_2 \cup P_3\}$-free graphs}

Choudum, Karthick and Shalu  described the structure of $\{ K_3, P_2 \cup P_3\}$-free graphs in the proof of Theorem 10 in \cite{cks} (see the following (P1)-(P6) and (P8)-(P10) in Lemma \ref{p2-p3 1}). Here, we show more on the structure of such graphs.

Let $\mathcal{H}$ be a family of imperfect, connected and $\{ K_3, P_2 \cup P_3\}$-free graphs. Then $C_{5} \sqsubseteq H$ for any $H \in \mathcal{H}$ by Theorem \ref{thm-p1}. Without loss of generality, let $V(C_5)= \{v_{1}, v_{2}, v_{3}, v_{4}, v_{5}\} = S_0$, where $v_{1}, v_{2}, v_{3}, v_{4}, v_{5}$ in clockwise order on the $C_5$. When there is no confusion, we omit the subscription $H$ in notations $d_{H}(x)$, $N_{H} (x)$ and so on. For any $x \in V(H) \setminus S_{0}$, denote $d(x, S_{0}) = \min \{d(x,y): y \in S_{0}\}$, and $S_{i}= \{x \in V(H) \setminus S_{0}| d(x, S_{0}) = i\} $ for $i \geq 1$. Then, for any $x \in S_1$, we have $H[N(x) \cap S_{0}]\cong K_{1}$ or $2K_{1}$ by $H$ is $K_{3}$-free. Define
 $$A_{i} = \{x \in S_{1}| N(x)\cap S_{0} = \{v_{i}\}\} $$
and $$B_{i} = \{x \in S_{1}| N(x)\cap S_{0} = \{v_{i- 1}, v_{i + 1}\}\},$$
for each $i \in [5]$. All arithmetic on indices in the remaining proof is done under modulo 5.

Next we show the structure of graph $H$ by the following Lemmas \ref{p2-p3 1}--\ref{p2-p3 5}.

\begin{lem}\label{p2-p3 1}

For every $H \in \mathcal{H}$ and every $i\in [5]$, the following $(P1)$--$(P11)$ hold.\\
$(P1)$ $|A_{i}| \leq 1$;\\
$(P2)$ $B_{i}$ is an independent set;\\
$(P3)$ $A_{i} \nsim A_{i + 1}$, $A_{i} \sim  A_{i + 2}$;\\
$(P4)$ $A_{i} \nsim  B_{i}$, $A_{i} \sim B_{i - 2} \cup B_{i + 2}$;\\
$(P5)$ $A_{i}\nsim  B_{i- 1} \cup B_{i+ 1}$;\\
$(P6)$ $B_{i} \nsim  B_{i + 2}$;\\
$(P7)$ for each $b \in B_i$ and $j \in \{i-1, i+1\}$, $|N(b) \cap B_{j}| \geq |B_{j}| -1 $;\\
$(P8)$ if $A_{i} \neq \emptyset$, then $B_{i-2} \nsim  B_{i + 2}$;\\
$(P9)$ $S_{j} = \emptyset$ when $j \geq 3$, and $S_{2}$ is an independent set;\\
$(P10)$ $A_{i} \nsim  S_{2}$;\\
$(P11)$ for each $z \in S_{2}$, $|N(z) \cap B_{i} | \leq 1$; for each $y \in B_{i}$, $|N(y) \cap S_2| \leq 1$; thus $\sum_ {j = 1}^{5}|B_{j}|\geq |S_{2}|$.
\end{lem}

\begin{proof}

(P1) If $|A_{i}| \geq 2$ for some $i \in [5]$, say $a, a' \in A_{i}$, then $a \nsim a'$ because $H$ is $K_3$-free. However, $H[\{v_{i-2}, v_{i+2}, a, v_i, a'\}] \cong P_{2} \cup P_{3}$, a contradiction.

Since $H$ is $K_3$-free, (P2), (P5) and (P6) hold. Moreover, (P3), (P4), and (P7) are true by $H$ being $P_{2} \cup P_{3} $-free.

(P8) By (P4), we see that $A_i \sim B_{i-2}$ $\cup$  $B_{i + 2}$. If $[B_{i-2},  B_{i + 2}] \neq \emptyset$, then $K_3 \sqsubseteq H[B_{i-2}\cup  B_{i + 2} \cup A_{i}]$, a contradiction.

(P9) If $S_{3} \neq \emptyset$, then there exist $z_{j} \in  S_{j}$ for $j = 2, 3$ such that $z_2 \sim z_3$, which means $H[\{z_{2}, z_{3}, v_{1}, v_{2}, v_{3}\}] \cong P_{2} \cup P_{3}$, a contradiction. Therefore, $S_{j} = \emptyset$, $j \geq 3$. If there exist two distinct vertices $z, z' \in S_{2}$ such that $z \sim z'$, then $H[\{z, z', v_{1}, v_{2}, v_{3} \}] \cong P_{2} \cup P_{3}$, a contradiction. Thus $S_{2}$ is an independent set when $S_2\neq \emptyset$.

(P10) If there exist $a \in A_{i}$ for some $i \in [5]$ and $z \in S_{2}$ such that $a\sim z$, then $H[\{a, z, v_{i + 1}, v_{i + 2},$ $ v_{i - 2}\}] \cong P_{2} \cup P_{3}$, a contradiction.

(P11) If there exists vertex $ z \in S_{2}$ such that $|N(z) \cap B_{i}| \geq 2$ for some $i \in [5]$, say two distinct vertices $y_{1},  y_{2} \in N(z) \cap B_{i}$, then $y_{1} \nsim y_{2}$ by (P2). But then $H[\{ v_{i - 2}, v_{i + 2}, y_{1}, z, y_{2} \}] \cong P_{2} \cup P_{3}$, a contradiction. So $|N(z) \cap B_{i} | \leq 1$ for each $i\in [5]$ and every $ z \in S_{2}$. If there exists $ y \in B_{i}$ such that $|N(y) \cap S_{2}| \geq 2$ for some $i \in [5]$, say two distinct vertices $z_{1}, z_{2} \in N(y) \cap S_{2}$, then $z_{1} \nsim z_{2}$ by (P9). But $H[\{v_{i - 2}, v_{i + 2}, z_{1},y, z_{2}\}] \cong P_{2} \cup P_{3}$, a contradiction. Thus $|N(y) \cap S_2| \leq 1$ for each $i\in [5]$ and any vertex $y \in B_{i}$. Furthermore, by (P10), we have $\sum_ {j = 1}^{5}|B_{j}|\geq |S_{2}|$.
\end{proof}

For simplicity, in the following, we omit ``in Lemma \ref{p2-p3 1}'' when we mention (P$j$) for $j \in [11]$. Let $h_A, h_B$ denote the number of non-empty sets in $\{A_j: 1\leq j\leq 5\}$, $\{B_j: 1\leq j\leq 5\}$, respectively. Then $0 \leq h_A, h_B \leq 5$. By (P1), $h_A = \sum_{j =1}^{5} |A_j|$. Moreover, for $i \in [5]$, let $A_i = \{a_i\}$ when $A_i\neq\emptyset$ by (P1), $b_{i} \in B_{i}$ when $B_i\neq\emptyset$, and let $z \in S_2$ when $S_2\neq\emptyset$.

\begin{definition}\label{B 1}
Let $\mathcal{B}_1$ denote the set of such graphs $H \in \mathcal{H}$ satisfying all of the following:

$\bullet$ $h_B =2$ and $S_2 \neq\emptyset$;

$\bullet$ for some $i \in [5]$, $B_{i}, B_{i+1}\neq\emptyset$, $[B_i, B_{i+1}] \neq \emptyset$ and $A_{i-2} = \emptyset$;

$\bullet$ $||B_i \setminus N(S_2)| - |B_{i+1} \setminus N(S_2)|| < |M|$, where $M \neq\emptyset$ is the set of missing edges between $N(S_2) \cap B_i$ and $N(S_2) \cap B_{i+1}$.
\end{definition}

\begin{definition}\label{B 2}
Let $\mathcal{B}_2$ denote the set of such graphs $H \in \mathcal{H}$ satisfying that $S_2 =\emptyset$, $h_A =0$, $h_B =5$, and for each $j\in [5]$, $[B_j, B_{j+1}] \neq \emptyset$.
\end{definition}

The following Lemmas \ref{p2-p3 2}--\ref{p2-p3 3} show more details on adjacent relation between $\bigcup_{j =1}^{5} B_j$ and $S_2$.
\begin{lem}\label{p2-p3 2}
When $S_{2} \neq \emptyset$, the following are true.

\begin{itemize}
  \item[\emph {(1)}] If $B_{i}, B_{i +2} \neq \emptyset$ and $[B_i \cup B_{i+2}, S_2] \neq \emptyset$  for some $i\in [5]$, then $|B_i|=|B_{i + 2}|=1$ and there exists vertex $z \in S_2$ such that $N(B_i) \cap S_2 = N(B_{i+2}) \cap S_2 = \{z\}$.

  \item[\emph {(2)}] If there exist $b_i \in B_{i}$ and $b_{i + 1} \in B_{i + 1}$ such that $b_i \nsim b_{i + 1}$ for some $i\in [5]$, then either $\{b_i, b_{i + 1}\} \nsim S_2$, or there exists vertex $z \in S_2$ such that $N(b_{i}) \cap S_2 = N(b_{i + 1}) \cap S_2 = \{z\}$.

  \item[\emph {(3)}] If $h_B=2$ and $B_i, B_{i+2} \neq \emptyset$  for some $i\in [5]$, or $B_{i-1}, B_i, B_{i+2}\neq\emptyset$ for some $i\in [5]$, then $|B_j|\leq 1$ for every $j \in [5]$, $\bigcup_{j =1}^{5} B_j$ is an independent set, and $S_2= \{z\} \sim \bigcup_{j =1}^{5} B_j$.
\end{itemize}
\end{lem}

\begin{proof}
(1)  By (P6), we see that $B_{i}\nsim B_{i + 2}$. W.l.o.g., assume that $[B_{i+2}, S_2] \neq \emptyset$, i.e. there exist $b_{i+2} \in B_{i+2}$ and $z \in S_2$ such that $b_{i+2} \sim z$. Then $\{z\} \sim B_i$, else $H[\{b_{i+2}, z, b_i,v_{i -1}, v_{i}\}] \cong P_{2} \cup P_{3}$ for some $b_i \in B_i$ satisfying that $z \nsim b_i$, a contradiction. Suppose $|B_{i+2}| \geq 2$, i.e. there exists $b_{i+2}' (\neq b_{i+2}) \in B_{i+2}$. Note that $|N(z) \cap B_{j} | \leq 1$ for every $j \in [5]$ and every $z \in S_2$ by (P11). By $b_{i+2} \sim z$, we have $b_{i+2}' \nsim z$. But then $H[\{b_i, z, v_{i+2}, v_{i-2}, b_{i+2}'\}] \cong P_{2} \cup P_{3}$, a contradiction. Thus $|B_{i+2}| =1$. If there exist two distinct vertices $b_i, b_i'\in B_i$, then $z \sim \{b_i, b_i'\}$ by $\{z\} \sim B_i$. But $H[\{v_{i-2}, v_{i+2}, b_i, z, b_{i}'\}] \cong P_{2} \cup P_{3}$, a contradiction. Thus $|B_i| =1$ and $B_i \cup B_{i+2} =\{b_i, b_{i+2}\} \sim z$. Furthermore, by (P11), we see that $|N(y) \cap S_2| \leq 1$ for every $j \in [5]$ and every $y \in B_{j}$. Thus there exists vertex $z \in S_2$ such that $N(B_i) \cap S_2 = N(B_{i+2}) \cap S_2 = \{z\}$.

(2) If one of $b_i, b_{i + 1}$ is anti-complete to $S_{2}$ and the other one has neighbors in $S_{2}$, w.l.o.g., let $\{b_i\}\nsim S_{2}$ and $b_{i + 1}\sim z$ for some $z \in S_2$, then $b_i \nsim z$. By $b_i \nsim b_{i+1}$, $H[ \{b_i, v_{i -1}, z, b_{i + 1}, $ $v_{i +2} \}] \cong P_{2} \cup P_{3}$, a contradiction. If both $b_i$ and $b_{i + 1}$ have neighbors in $S_{2}$, then $|N(b_i) \cap S_{2}|=|N(b_{i + 1}) \cap S_{2}|=1$ by (P11). Furthermore, if $N(b_i) \cap S_{2}\neq N(b_{i + 1}) \cap S_{2}$, say $N(b_i)\cap S_{2}= \{p\} $ and $ N(b_{i + 1}) \cap S_{2}= \{q\}$, then $p \neq q$, $b_{i + 1} \nsim p$ and $b_i \nsim q$. Moreover, $p \nsim q$ by (P9). By $b_i \nsim b_{i+1}$, $H[\{p, b_i, v_{i + 2}, b_{i + 1},q \}] \cong P_{2} \cup P_{3}$, a contradiction. Thus, either $\{b_i, b_{i + 1}\} \nsim S_2$, or there exists vertex $z\in S_2$ such that $N(b_i) \cap S_2 = N(b_{i + 1}) \cap S_2 = \{z\}$.

(3) If $h_B=2$ and $B_i, B_{i+2} \neq \emptyset$ for some $i\in [5]$, then $|B_i|= |B_{i+2}|= |S_{2}| = 1$ and $S_2 = \{z\} \sim B_i \cup B_{i+2}$ by $S_{2} \neq \emptyset$, (P10)-(P11) and Lemma \ref{p2-p3 2} $(1)$. Furthermore, $B_i \cup B_{i+2}$ is an independent set  by (P6).

If $B_{i-1}, B_i, B_{i+2} \neq\emptyset$ for some $i\in [5]$, then $[B_j, S_2] \neq \emptyset$ for some $j \in [5]$ by $S_{2} \neq \emptyset$ and (P10). Moreover, $B_{j-2}\cup B_{j+2} \neq \emptyset$. W.l.o.g., assume $B_{j+2} \neq \emptyset$. By Lemma \ref{p2-p3 2} $(1)$, $|B_{j}| =|B_{j+2}| = 1$ and there exists a vertex $z \in S_2$ such that $N(B_{j}) \cap S_2= N(B_{j +2}) \cap S_2 = \{z\}$. Since $B_{i-1}, B_i, B_{i+2} \neq\emptyset$ for some $i\in [5]$, we see that $B_{j-1} \cup B_{j-2} \neq\emptyset$. W.l.o.g., assume $B_{j-1} \neq \emptyset$. Then $|B_{j-1}| =|B_{j}|= |B_{j+2}| =1$ and $ N(B_{j -1}) \cap S_2 =N(B_{j}) \cap S_2= N(B_{j +2}) \cap S_2 = \{z\}$ by Lemma \ref{p2-p3 2} $(1)$. Next we consider the remaining two sets $B_{j +1}$ and $ B_{j-2}$. If $B_{j +1}\cup B_{j-2}= \emptyset$, then $S_2= \{z\}$ and $S_2 \sim B_{j-1}\cup B_{j}\cup B_{j+2}$. Recalling that $H$ is $K_3$-free, $B_{j-1}\cup B_{j}\cup B_{j+2}$ is an independent set of size three. If $B_{j +1}\cup B_{j-2} \neq \emptyset$, then $|B_j|\leq 1$ for every $j \in [5]$ and $\{z\} = S_2 \sim \bigcup_{j =1}^{5} B_j$ by Lemma \ref{p2-p3 2} $(1)$. Furthermore, $\bigcup_{j =1}^{5} B_j$ is an independent set because $H$ is $K_3$-free.
\end{proof}

\begin{lem}\label{p2-p3 3}
If one of the following is satisfied:
\begin{itemize}
  \item[\emph {(1)}] $h_B=2$, $B_{i-1}, B_i \neq \emptyset$ and $B_{i-1} \nsim B_{i}$ for some $i\in [5]$;

  \item[\emph {(2)}] $B_{i-1}, B_i, B_{i+1} \neq \emptyset$,  and $B_{i-1} \nsim B_{i}$ or $B_i \nsim B_{i+1}$, for some $i\in [5]$;
\end{itemize}
then $|B_j|\leq 1$ for each $j\in[5]$, $\bigcup_{j =1}^{5} B_j$ is an independent set, $|S_2| \leq 1$ and $S_2 \sim \bigcup_{j =1}^{5} B_j$.
\end{lem}

\begin{proof}
(1) By (P7), $|B_{i-1}|=|B_{i}|=1$ and $B_{i-1}\cup B_{i}$ is an independent set. Furthermore, $|S_2|\leq1$ and $S_2 \sim B_{i-1}\cup B_{i}$ by (P10)-(P11) and Lemma \ref{p2-p3 2} $(2)$.

(2) Firstly, we show that if $B_{i-1}, B_i, B_{i+1} \neq \emptyset$, and $B_{i-1} \nsim B_{i}$ ($B_{i} \nsim B_{i+1}$) for some $i\in [5]$, then $B_{i} \nsim B_{i+1}$ ($B_{i-1} \nsim B_{i}$). To the contrary, assume that there exist $b_i \in B_i$ and $b_{i+1} \in B_{i+1}$ such that $b_i \sim b_{i+1}$. Then, for any $b_{i-1} \in B_{i-1}$, we have that $b_{i-1} \nsim b_i$ by $B_{i-1} \nsim  B_{i}$, and $b_{i-1} \nsim b_{i+1}$ by (P6). But $H[\{b_{i-1}, v_{i -2}, v_{i+1}, b_i, b_{i+1}\}] \cong P_{2} \cup P_{3}$, a contradiction. Thus, if $B_{i-1} \nsim B_{i}$, then $B_{i} \nsim B_{i+1}$. Analogously, if $B_{i} \nsim B_{i+1}$, then $B_{i-1} \nsim B_{i}$. By (P7), $|B_{i-1}| = |B_i|= |B_{i +1 }| = 1$. Furthermore, $B_{i-1} \cup B_i \cup B_{i +1 }$ is an independent set by (P6).

If $h_B \geq 4$, applying the above  discussion on $B_{i}, B_{i+1}, B_{i+2}$ or $B_{i-2}, B_{i-1}, B_{i}$, then $|B_j|\leq 1$ for each $j\in[5]$ and $\bigcup_{j =1}^{5} B_j$ is an independent set by (P6). By (P10)-(P11) and Lemma \ref{p2-p3 2} $(1)$-$(2)$, $|S_2|\leq1$ and $S_2 \sim \bigcup_{j =1}^{5} B_j$.
\end{proof}

The following two lemmas describe some relations between $\bigcup_{j =1}^{5} B_j$ and $ (\bigcup_{j =1}^{5} A_j) \cup S_2$.

\begin{lem}\label{p2-p3 4}

Suppose that $B_{i}, B_{i +2} \neq \emptyset$ for some $i\in [5]$. If $A_{i}\cup A_{i +2} \neq \emptyset$ or $A_{i-2}, A_{i-1} \neq \emptyset$, then $|B_{i}|=|B_{i + 2}|=1$.
\end{lem}

\begin{proof}
For simplicity, for $s \in \{i, i+2\}$, if $|B_{s}|\geq 2$, let $b_s, b_s'\in B_{s}$ and $b_s \neq b_s'$.

We firstly show that if $A_{i}\cup A_{i +2} \neq \emptyset$, then $|B_{i}|=|B_{i + 2}|=1$. W.l.o.g., assume $A_{i} \neq \emptyset$. If $|B_{i}|\geq 2$, then $\{a_{i}\} \nsim \{b_i, b_i'\}$ and $a_{i} \sim b_{i+2}$ by (P4). Moreover, by (P2) and (P6), $\{b_i, b_i', b_{i+2}\}$ is an independent set. But then $H[\{a_{i}, b_{i+2}, b_i, v_{i-1}, b_i'\}] \cong P_{2} \cup P_{3}$, a contradiction. Thus $|B_{i}|=1$. If $|B_{i+2}|\geq 2$, then $\{a_{i}\} \sim \{b_{i+2}, b_{i+2}'\}$ and $a_{i} \nsim b_{i}$ by (P4). By (P2) and (P6), $\{b_i, b_{i+2}, b_{i+2}'\}$ is an independent set. But $H[\{b_{i}, v_{i-1}, b_{i+2}, a_{i}, b_{i+2}'\}] \cong P_{2} \cup P_{3}$, a contradiction. Thus $|B_{i+2}|=1$.

We next show that if $A_{i-2}, A_{i-1} \neq \emptyset$, then $|B_{i}|=|B_{i + 2}|=1$. Suppose $|B_{i}|\geq 2$. Then $\{a_{i-2}\} \sim \{b_i, b_i'\}$ and $a_{i-1} \sim b_{i+2}$ by (P4). Moreover, $\{a_{i-1}\} \nsim \{b_i, b_i' \}$ and $a_{i-2} \nsim b_{i+2}$ by (P5), and $a_{i-1} \nsim a_{i-2}$ by (P3). Note that $\{b_i, b_i', b_{i+2}\}$ is an independent set. Then $H[\{a_{i-1}, b_{i+2}, b_i, a_{i-2}, b_i'\}] \cong P_{2} \cup P_{3}$, a contradiction. Thus $|B_{i}|=1$. Analogously, we have $|B_{i+2}|=1$.
\end{proof}

\begin{lem}\label{p2-p3 5}
If one of the following holds:

\begin{itemize}
  \item[\emph {(1)}] $h_B=2$, $B_{i}, B_{i+2} \neq\emptyset$ and $h_A \geq 3$, for some $i \in [5]$,
  \item[\emph {(2)}] $h_B=3$, $B_{i-2}, B_{i}, B_{i+2} \neq\emptyset$ and $A_{i-2}\cup A_{i}\cup  A_{i+2} \neq \emptyset$, for some $i \in [5]$,
  \item[\emph {(3)}] $h_B \geq 4$ and $ h_A \geq 1$,
\end{itemize}
then $|B_j| \leq 1$ for every $j \in [5]$, $\bigcup_{j =1}^{5} B_j$ is an independent set, $|S_2| \leq 1$ and $S_2 \sim \bigcup_{j =1}^{5} B_j$.
\end{lem}

\begin{proof}

(1) By $h_A\geq 3$, there is either $A_{i}\cup A_{i +2} \neq \emptyset$ or $A_{i-2}, A_{i-1}, A_{i+1}\neq \emptyset$. Then $|B_{i}|=|B_{i + 2}|=1$ and $B_{i} \cup B_{i + 2}$ is an independent set by Lemma \ref{p2-p3 4} and (P6). By Lemma \ref{p2-p3 2} $(3)$, we have $|S_2| \leq 1$ and $S_2 \sim \bigcup_{j =1}^{5} B_j$.

(2) If $S_2 \neq \emptyset$, then we have done by Lemma \ref{p2-p3 2} $(3)$. Next we assume $S_2 = \emptyset$. If $A_{i}\neq \emptyset$, then $|B_{i-2}|=|B_i|= |B_{i+2}| =1$ by Lemma \ref{p2-p3 4}, and $B_{i-2} \nsim B_{i+2}$ by (P8). Furthermore, $B_{i-2} \cup B_i \cup B_{i+2}$ is an independent set by (P6). If $A_{i-2}\neq \emptyset$, then $|B_i|=|B_{i-2}|= 1$ by Lemma \ref{p2-p3 4}. We claim that $B_{i-2} \nsim B_{i+2}$. (Otherwise, there exist vertices $b_{i-2} \in B_{i-2}$ and $b_{i+2} \in B_{i+2}$ such that $b_{i-2} \sim b_{i+2}$. By (P4), $a_{i-2} \nsim b_{i-2}$ and $a_{i-2} \sim b_i$. By (P5), $a_{i-2} \nsim b_{i+2}$. Moreover, by (P6), $\{b_i\} \nsim \{b_{i-2},  b_{i+2}\}$. But then $H[\{a_{i-2}, b_i, v_{i+2}, b_{i-2}, b_{i+2}\}] \cong P_{2} \cup P_{3}$, a contradiction.) Then $|B_{i-2}|=|B_{i+2}| = 1$ by (P7). Thus, $|B_i| =|B_{i-2}|=|B_{i+2}| = 1$ and $B_{i-2} \cup B_i \cup B_{i+2}$ is an independent set by (P6). Analogously, if $A_{i+2}\neq \emptyset$, we also have that $|B_i| =|B_{i-2}|=|B_{i+2}| = 1$ and $B_{i-2} \cup B_i \cup B_{i+2}$ is an independent set.

(3) By  $ h_A \geq 1$, let $A_i \neq \emptyset$ for some $i\in [5]$. By $h_B \geq 4$, $B_{i-2} \cup B_{i+2} \neq\emptyset$. If $B_{i-2}, B_{i+2} \neq \emptyset$, then $B_{i-2}\nsim B_{i+2}$ by $A_{i} \neq \emptyset$ and (P8). We next show that if $B_{i-2} = \emptyset$, then $B_{i-1}\nsim B_{i}$. Otherwise, there exist $b_{i-1} \in B_{i-1}$  and $b_i \in B_i$ such that $b_{i-1} \sim b_{i}$. By (P4), $a_i \sim b_{i+2}$ and $a_i \nsim b_i$; by (P5), $a_i \nsim  b_{i-1}$. Moreover, by (P6), $b_{i+2} \nsim \{b_i, b_{i-1}\}$. But then $H[\{a_i, b_{i+2}, b_{i-1}, b_{i}, v_{i-1}\}] \cong P_{2} \cup P_{3}$, a contradiction. Analogously, if $B_{i+2} = \emptyset$, then $B_{i}\nsim B_{i+1}$. In any of the three cases, by Lemma \ref{p2-p3 3} (2), we have done.
\end{proof}

\section{Proof of Theorem \ref{thm-p2p31}}

Let $G$ be a $\{3K_1, \overline{P_2 \cup P_3}\}$-free graph on $n$ vertices. Next we focus on the complement $H = \overline{G}$ of $G$, and show that $\alpha(H) \geq (n-1)/3$. If $H$ is perfect, then $G$ is perfect by Theorem \ref{thm-p2}, we have done. So we may assume that $H$ is imperfect.

Let $\mathcal{H}$ be a family of imperfect, connected and $\{K_3, P_2 \cup P_3\}$-free graphs. First, we discuss the case that $H \in \mathcal{H}$. Then $C_5 \sqsubseteq H$ by Theorem \ref{thm-p1}. Let $v_i$, $A_i$, $B_i$, $S_2$, $h_A$ and $h_B$ be defined as in Section 3, $i \in [5]$. By (P9), $V(H) = \bigcup_{i =1}^{5} (A_i \cup B_i \cup \{v_i\}) \cup S_2$. Next we discuss four cases according to the value of $h_B$.

If $h_B \leq 2$ and $B_{i-2} \cup B_{i-1} \cup B_{i+1} = \emptyset$ for some $i \in [5]$, then $X_{2,1}= B_{i} \cup B_{i+2} \cup \{v_{i}, v_{i + 2}\} \cup A_{i+1}$, $Y_{2,1}= A_{i-2} \cup A_{i+2} \cup S_2 \cup \{v_{i-1}, v_{i + 1}\}$ and $Z_{2,1}= A_{i-1} \cup A_{i} \cup \{ v_{i - 2}\}$ are three vertex-disjoint independent sets by (P1)-(P3), (P6) and (P9)-(P10). If $h_B =2$ and $B_{i}, B_{i+1}\neq\emptyset$ for some $i \in [5]$, then we consider two subcases. If $B_i \nsim B_{i+1}$, then $ n\leq 13$ by Lemma \ref{p2-p3 3}. Note that $X_{2,2}= B_{i} \cup B_{i+1} \cup \{ v_{i + 2}\} \cup  A_{i} \cup A_{i+1}$, $Y_{2,2}= A_{i-2} \cup A_{i+2} \cup S_2 \cup \{v_{i-1}, v_{i + 1}\}$ and $Z_{2,2}= A_{i-1} \cup \{v_{i},  v_{i + 2}\}$ are three vertex-disjoint independent sets by (P1)-(P5) and (P9)-(P10). If $[B_i, B_{i+1}] \neq\emptyset$, then $A_{i-2}=\emptyset$ by (P8). By (P1)-(P3), (P5) and (P9)-(P10), we see that $X_{2,3}= B_{i} \cup A_{i-1} \cup \{v_{i}, v_{i + 2}\}$, $Y_{2,3}= B_{i+1} \cup A_{i+2} \cup \{v_{i-2}, v_{i + 1}\}$ and $Z_{2,3}= A_{i} \cup A_{i+1} \cup S_2 \cup \{v_{i-1}\}$ are three vertex-disjoint independent sets. In any of the three cases, we have $|X_{2,j} \cup Y_{2,j} \cup Z_{2,j}| =n$ for $j \in [3]$. Thus $\alpha(H) \geq  n/3$ when $h_B \leq 2$.

If $h_B =3$ and $B_{i}, B_{i+1}, B_{i+2} \neq \emptyset$ for some $i \in [5]$, then $X_{3,1}= B_{i} \cup B_{i+2} \cup \{v_{i}, v_{i + 2}\} \cup A_{i+1}$, $Y_{3,1}= B_{i+1} \cup A_{i} \cup  \{ v_{i - 2}\}$ and $Z_{3,1}= A_{i-2} \cup A_{i+2} \cup S_2 \cup \{v_{i-1}, v_{i + 1}\}$ are three vertex-disjoint independent sets by (P1)-(P3), (P5)-(P6) and (P9)-(P10). We consider two subcases. When $[B_{i+1}, B_{i+2}] \neq\emptyset$, there is $A_{i-1}=\emptyset$ by (P8), which means that $|X_{3,1} \cup Y_{3,1} \cup Z_{3,1}| =n$ and then $\alpha(H) \geq  n/3$. When $B_{i+1}\nsim B_{i+2}$, we know that $ n \leq 14$ and $|S_2| \leq 1$ by Lemma \ref{p2-p3 3}. Furthermore, if $n \leq 12$, then $\alpha(H) \geq |X_{3,1}| \geq 4\geq n/3$; if $13 \leq n \leq 14$, then $h_A + |S_2|\geq 5$, which implies that $\alpha(H) \geq \max \{|X_{3,1}|,|Z_{3,1}|\} \geq 5> n/3$. If $h_B =3$ and $B_{i-2}, B_{i}, B_{i+2}\neq\emptyset$ for some $i \in [5]$, then we consider two subcases. When $A_{i-2} \cup A_{i} \cup A_{i+2}\neq\emptyset$, we have $ n \leq 14$ and $|B_{i-2}|=| B_{i+2}| =1$ by Lemma \ref{p2-p3 5}. Furthermore, by (P1)-(P5), $X_{3,2}= B_{i-2} \cup B_{i+2} \cup \{v_{i}\} \cup A_{i-2} \cup A_{i + 2}$ is an independent set of size five. Thus $\alpha(H) \geq |X_{3,2}| > n/3$. When $A_{i-2} \cup A_{i} \cup A_{i+2}= \emptyset$, we can see that $X_{3,3}= B_{i} \cup B_{i+2} \cup \{v_{i}, v_{i + 2}\} \cup A_{i+1}$, $Y_{3,3}= B_{i-2} \cup A_{i-1} \cup \{v_{i-2}, v_{i + 1}\}$ and $Z_{3,3}= S_{2} \cup \{ v_{i - 1}\}$ are three vertex-disjoint independent sets by (P1)-(P2), (P5)-(P6) and (P9). Noting that $|X_{3,3} \cup Y_{3,3} \cup Z_{3,3}| =n$, $\alpha(H) \geq  n/3$.

If $h_B =4$, say $B_{i}= \emptyset$ for some $i \in [5]$, then $X_{4,1}= B_{i-2} \cup B_{i+1} \cup \{v_{i-2}, v_{i + 1}\} \cup A_{i+2}$, $Y_{4,1}= B_{i-1} \cup B_{i+2} \cup \{v_{i-1}, v_{i + 2}\} \cup A_{i-2}$ and $Z_{4,1}= S_2 \cup \{v_{i}\}$ are three vertex-disjoint independent sets by (P1)-(P2), (P5)-(P6) and (P9). If $h_A =0$, then $|X_{4,1} \cup Y_{4,1} \cup Z_{4,1}| =n$, which means that $\alpha(H) \geq  n/3$. If $h_A \geq 1$, then $ n \leq 15$ and $|S_2| \leq 1$ by Lemma \ref{p2-p3 5}. Clearly, $\alpha(H) \geq |X_{4,1}| \geq 4\geq n/3$ when $n \leq 12$. If $ 13\leq n \leq 15$, there is $h_A + |S_2|\geq 4$. Moreover, when $S_2 = \emptyset$, $\alpha(H) \geq \max\{|X_{4,1}|, |Y_{4,1}|\} \geq 5\geq n/3$; when $S_2 \neq \emptyset$, $h_A \geq 3$, which means that there exists an $i \in [5]$ such that $\alpha(H) \geq |A_{i} \cup A_{i+1} \cup S_2 \cup \{v_{i-1}, v_{i + 2}\}|=5 \geq  n/3$.

If $h_B =5$, then we consider the following two subcases. If $S_2=\emptyset$ and $h_A=0$, then $X_{5,1}= B_{i-2} \cup B_{i+1} \cup \{v_{i-2}, v_{i + 1}\}$, $Y_{5,1}= B_{i-1} \cup B_{i+2} \cup \{v_{i-1}, v_{i + 2}\}$ and $Z_{5,1}= B_i \cup \{v_{i}\}$ are three vertex-disjoint independent sets by (P2) and (P6), which means that $\alpha(H) \geq n/3$. If $S_2\neq\emptyset$, or $S_2=\emptyset$ and $h_A\geq 1$, then $n \leq 16$ and $\bigcup_{i=1}^{5} B_i$ is an independent set of size five by Lemma \ref{p2-p3 2} $(3)$ and Lemma \ref{p2-p3 5}. Clearly, when $n \leq 15$, there is $\alpha(H) \geq 5 \geq n/3$; when $n = 16$, $H \cong H^{*}$ (see Figure 2) and $\alpha(H)\geq (n-1)/3$.

If $H$ is an imperfect disconnected graph, then $H$ is a union of a graph $H' \in \mathcal{H}$ and $n-|H'|$ isolated vertices (see Remark). It means that $|H'| \leq n -1$ and then $\omega(G)= \alpha(H) \geq\alpha(H')  +  n- |H'| \geq (|H'|-1)/3 + n-|H'| > n/3$.

Totally, $\omega(G) \geq  n/3$ when $G$ is $\{3K_1, \overline{P_2 \cup P_3}\}$-free except the case that when $G \cong \overline{H^*}$, $\omega(G) \geq (n-1)/3$. By the above discussion, we only need to show that $\alpha(H^{*})\leq (n-1)/3 = 5$. Note that $H^{*}[\bigcup_{i=1}^{5} \{v_i\}] \cong C_5$ and $H^{*}[\bigcup_{i=1}^{5} \{a_i\}] \cong C_5$. Suppose that $Q$ is a maximum independent set in $H^{*}$. If $ z \in Q$, then $Q \cap \bigcup_{i=1}^{5} \{b_i\} = \emptyset$, which means that $|Q|  \leq 5$. If $z \notin Q$, then $|Q| \leq 5$ when $|\bigcup_{i=1}^{5} \{b_i\} \cap Q| \leq 1$. The remaining cases to be considered are that $z \notin Q$ and $|\bigcup_{i=1}^{5} \{b_i\} \cap Q| \geq 2$.
By (P4) and the definition of $B_i$ $(i \in [5])$, $Q \cap \bigcup_{i=1}^{5} \{v_i\}=Q \cap \bigcup_{i=1}^{5} \{a_i\}= \emptyset$ when $|\bigcup_{i=1}^{5} \{b_i\} \cap Q| \geq4$; $Q \cap \bigcup_{i=1}^{5} \{v_i\} = \emptyset$ and $|Q \cap \bigcup_{i=1}^{5} \{a_i\}| \leq 1$ when $ |\bigcup_{i=1}^{5} \{b_i\} \cap Q| = 3$ and the three $b_i$'s in $Q$ are consecutive; $|Q \cap \bigcup_{i=1}^{5} \{v_i\}| \leq 1$  and $Q \cap \bigcup_{i=1}^{5} \{a_i\} = \emptyset$ when $ |\bigcup_{i=1}^{5} \{b_i\} \cap Q| = 3$ and the three $b_i$'s in $Q$ are non-consecutive; $|Q \cap \bigcup_{i=1}^{5} \{v_i\}| \leq 1$ and $|Q \cap \bigcup_{i=1}^{5} \{a_i\}| \leq 2$ when $ |\bigcup_{i=1}^{5} \{b_i\} \cap Q| = 2$ and the two $b_i$'s in $Q$ are consecutive; $|Q \cap \bigcup_{i=1}^{5} \{v_i\}| \leq 2$ and  $|Q \cap \bigcup_{i=1}^{5} \{a_i\}| \leq 1$ when $ |\bigcup_{i=1}^{5} \{b_i\} \cap Q| = 2$ and the two $b_i$'s in $Q$ are non-consecutive. In any cases, there is $\alpha(H^{*})=|Q|  \leq5$.

\section{Proof of Theorem \ref{thm-paraglider-boundedness}}

Before proceeding with the proof, we give one more definition. The clique covering number of graph $F$, denoted by $\theta(F)$, is the minimum number of disjoint cliques required to cover vertex set $V(F)$. Clearly, $\theta(F) = \chi(\overline{F})$.

Next, for any $\{3K_1, \overline{P_2 \cup P_3}\}$-free graph $G$, we focus on the complement $H = \overline{G}$ of $G$, and show that $\theta(H) \leq \max\{\alpha(H)+3, 2\alpha(H)-2\}$. If $H$ is perfect, then $G$ is perfect by Theorem \ref{thm-p2}, we have done. So we may assume that $H$ is imperfect.

Let $\mathcal{H}$ be a family of imperfect, connected and $\{K_3, P_2 \cup P_3\}$-free graphs. First, we discuss the case that $H \in \mathcal{H}$. Then $C_5 \sqsubseteq H$ by Theorem \ref{thm-p1}. Let $v_i$, $A_i$, $B_i$, $S_2$, $h_A$ and $h_B$ be defined as in Section 3, $i \in [5]$. By (P9), $V(H) = \bigcup_{i =1}^{5} (A_i \cup B_i \cup \{v_i\}) \cup S_2$. Next we consider four cases.

\noindent {\bf Case 1.} $h_B \leq 2$ and $B_{i-2} \cup B_{i-1} \cup B_{i+1} = \emptyset$ for some $i \in [5]$.

By (P2), (P6) and (P9), $X = B_{i} \cup B_{i+2} \cup \{v_{i}, v_{i + 2}\}$ and $Y =S_2 \cup \{v_{i-1}, v_{i + 1}\}$ are two vertex-disjoint independent sets. By (P11), for any vertex $b \in B_i\cup B_{i+2}$, $|N(b) \cap S_2| \leq 1$; when $S_2 \neq \emptyset$, for any vertex $z \in S_2$, $|N(z) \cap B_i|= 1$ or $|N(z) \cap B_{i+2}| = 1$. Thus we may pick a matching $M$ of size $|S_2| +2$ to cover $Y $, $\{v_{i}, v_{i + 2}\}$ and $|S_2|$ vertices in $B_{i} \cup B_{i+2}$. For the set $X \cup Y$, the vertices in $X \setminus V(M)$ are not covered by $M$. Combining the independent set $X \setminus V(M)$, there are $|X|$ disjoint $K_1$'s and $K_2$'s to cover $X \cup Y$, i.e. $\theta_{1} =\theta(H[ X\cup Y ]) \leq |X| \leq \alpha(H)$. Moreover, by (P1) and (P3), $\theta_{2} =\theta(H[ \bigcup_{j=1}^{5} A_j \cup \{v_{i-2}\} ]) \leq 3$. Thus $\theta(H)\leq \theta_{1} + \theta_{2} \leq \alpha(H) + 3$.

\noindent {\bf Case 2.} $h_B =2$ and $B_{i}, B_{i+1}\neq\emptyset$ for some $i \in [5]$.

Let $L_i= N(S_2) \cap B_i$, $R_i =B_i \setminus L_i$, $L_{i+1}= N(S_2) \cap B_{i+1} $ and $R_{i+1} = B_{i+1} \setminus L_{i+1}$. Then $S_2 \cup R_i $ and $S_2 \cup R_{i+1}$  are two independent sets by (P2) and (P9). For vertices $x \in L_i$ and $y\in L_{i+1}$, when $x \sim y$, $x$ and $y$ have no common neighbors in $S_2$ because $H$ is $K_3$-free; when $x \nsim y$, there exists a vertex $z \in S_2$ such that $N(x) \cap S_2 = N(y) \cap S_2= \{z\} $ by Lemma \ref{p2-p3 2} (2). Let $M$ be the set of missing edges between $L_i$ and $L_{i+1}$. By (P7), $M$ is an empty set or a matching. Then
\begin{equation}\label{case2,M}
\begin{aligned}
0 \leq |M| \leq \min \{|L_i|, |L_{i+1}|\}.
\end{aligned}
\end{equation}
Moreover, by (P11), $[L_i, S_2]$ and $[L_{i+1}, S_2]$ are two matchings. Thus $|M| \leq |S_2|$ and
\begin{equation}\label{case2,S2-M}
\begin{aligned}
|S_2| =|L_i \cup L_{i+1}| -|M|.
\end{aligned}
\end{equation}
Let $A \in \{A_{i-1}, A_{i+1}\}$ and $A' \in \{A_{i}, A_{i+2}\}$.
By (P5) and (P10), we may find two independent sets
\begin{equation}\label{case2,xy}
\begin{aligned}
X = R_i \cup  S_2 \cup \{v_i, v_{i+2}\} \cup A, \ \ \  Y = R_{i+1} \cup  S_2 \cup \{v_{i+1}, v_{i-1}\} \cup A'.
\end{aligned}
\end{equation}
By \eqref{case2,M} and \eqref{case2,S2-M}, we have
\begin{equation}\label{hb=2}
\begin{aligned}
\alpha(H)&\geq \max \{|X|, |Y|\}\\
 &\geq \max \{|R_i \cup A|, |R_{i+1}  \cup A'|\}  + |S_2|+  2\\
&= \max \{|B_i \cup L_{i+1} \cup A| -|M|, |B_{i+1} \cup L_i \cup A'| -|M|\} +2\\
&\geq \max \{|B_i\cup A|, |B_{i+1}\cup A'|\} +2.
\end{aligned}
\end{equation}
Next we consider three subcases.

\noindent {\bf Case 2.1.}  $[B_i, B_{i+1}] \neq \emptyset$ and $S_2 \neq\emptyset$.

Then $A_{i-2} = \emptyset$ by (P8). Recalling that $[L_i, S_2]$ and $[L_{i+1}, S_2]$ are two matchings, $[L_i\cup (L_{i+1} \setminus V(M)), S_2 ]$ is a matching of size $|S_2|$. Thus
\begin{equation}\label{case2,m1}
\begin{aligned}
\theta(H[L_i \cup (L_{i+1} \setminus V(M))\cup S_2]) \leq  |S_2|.
\end{aligned}
\end{equation}
W.l.o.g., let $|R_i|\geq |R_{i+1}|$. Next we show that $\theta(H) \leq \max\{\alpha(H) + 3, 2\alpha(H)-2-|R_i|\}$.

When $|R_i|\geq |R_{i+1}| + |M|$, by (P7), we may pick a matching $M_1$ of size $|R_{i+1}| + |M|$ from $[R_i, R_{i+1} \cup (L_{i+1} \cap V(M))]$ to cover $|R_{i+1}| + |M| $ vertices in $R_i$, all vertices in $R_{i+1} \cup (L_{i+1} \cap V(M))$. Combining the independent set $R_i \setminus V(M_1)$ of size $|R_i|- (|R_{i+1}| + |M|)$ and two $K_2$'s, $v_{i-1}v_i$, $v_{i+1}v_{i+2}$, by \eqref{case2,xy} and \eqref{case2,m1}, $\theta_{1} =\theta(H[ B_i \cup B_{i+1} \cup S_2 \cup \bigcup_{j=1}^{5} \{v_j\} \setminus \{v_{i-2}\}]) \leq  |S_2| + |M_1| + |R_i \setminus V(M_1)| +2 =| S_2 |+ |R_i | +2 \leq |X| \leq \alpha(H)$. Noting that $\theta_{2} =\theta(H[ \{v_{i-2}\} \cup \bigcup_{j=1}^{5} A_j]) \leq 3$, we have $\theta(H)\leq \theta_{1} + \theta_{2} \leq \alpha(H) + 3$.

When $|R_{i+1}| \leq |R_i| \leq |R_{i+1}| + |M|-1$, this is exactly the case that $H \in \mathcal{B}_1$. Then we may pick a matching $M_2$ of size $|R_i|$ from $[R_i, R_{i+1} \cup (L_{i+1} \cap V(M))]$ to cover $R_i $ and $|R_i|$ vertices in $ R_{i+1}\cup (L_{i+1} \cap V(M))$. Noting that $|R_{i+1} \cup (L_{i+1} \cap V(M))| = |R_{i+1}| + |M|$, there is a vertex $b_{i+1} \in R_{i+1} \cup (L_{i+1} \cap V(M))$ uncovered by $M_2$. Thus, we may pick three disjoint $K_2$'s, $v_ib_{i+1}, v_{i-1}v_{i-2}, v_{i+1}v_{i+2},$ to cover $\{b_{i+1}\} \cup \bigcup_{j=1}^{5} \{v_j\}$. Combining the independent set $I= R_{i+1} \cup (L_{i+1} \cap V(M))\setminus (\{b_{i+1}\} \cup V(M_2))$ of size $|R_{i+1}|+ |M| - |R_i| -1$, by \eqref{case2,m1}, we have
 \begin{equation}\label{case2,3-2}
\begin{aligned}
 \theta_3 &=\theta(H[B_i \cup B_{i+1} \cup S_2 \cup \bigcup_{j=1}^{5} \{v_j\}]) \\
 &\leq |S_2| + |M_2| + 3 + |I|\\
 &= |S_2| + |R_i| + 3 + |R_{i+1}|+ |M| - |R_i| -1\\
 &=|S_2| + |R_{i+1}|+ |M| +2.
 \end{aligned}
\end{equation}
By $A_{i-2} = \emptyset$, \eqref{case2,S2-M} and \eqref{case2,3-2}, we have
\begin{equation}\label{hb=2,2}
\begin{aligned}
\theta(H)&\leq \theta_3+ \theta(H[\bigcup_{j=1}^{5} A_j \setminus A_{i-2} ]) \\
&\leq |S_2| + |R_{i+1}| + |M| +2 + \theta(H[\bigcup_{j=1}^{5} A_j \setminus A_{i-2}])\\
&= |B_{i} \cup B_{i+1}|-|R_i| + 2 + \theta(H[\bigcup_{j=1}^{5} A_j \setminus A_{i-2}]).
\end{aligned}
\end{equation}
Note that $A_{i-1}\sim A_{i+1}$ and $A_{i}\sim A_{i+2}$ by (P3). By \eqref{hb=2}, if $A_{i-1}\cup  A_{i+1}=\emptyset$ and $|A_{i}\cup  A_{i+2}|\geq 1$, or $A_{i}\cup  A_{i+2}=\emptyset$ and $|A_{i-1}\cup  A_{i+1}|\geq 1$, then $\theta(H[\bigcup_{j=1}^{5} A_j \setminus A_{i-2}]) =1$ and  $|B_{i} \cup B_{i+1}| \leq 2\alpha(H) - 5$; if $|A_{i}\cup  A_{i+2}|\geq 1$ and  $|A_{i-1}\cup  A_{i+1}|\geq 1$, then $\theta(H[\bigcup_{j=1}^{5} A_j \setminus A_{i-2}]) \leq2$ and $|B_{i} \cup B_{i+1}| \leq 2\alpha(H) - 6$; if $h_A =0$, then $\theta(H[\bigcup_{j=1}^{5} A_j \setminus A_{i-2}]) =0$ and $|B_{i} \cup B_{i+1}| \leq 2\alpha(H) - 4$. In any of cases, by \eqref{hb=2,2}, we have $\theta(H) \leq 2\alpha(H)-2 -|R_i|=2\alpha(H)-2 -\max\{|R_i|, |R_{i+1}|\}$.

\noindent {\bf Case 2.2.} $[B_{i}, B_{i+1}] \neq \emptyset$ and $S_2 =\emptyset$.

Then $T_1 =B_i \cup \{v_i, v_{i+2}\}$ and $T_2 =B_{i+1} \cup  \{v_{i+1}, v_{i-1}\}$ are two vertex-disjoint independent sets. W.l.o.g., let $|B_i| \geq |B_{i+1}|$. By (P7), we may pick a matching $M_3$ of size $ |B_{i+1}| +2$ to cover all vertices in $T_2$, $\{v_i, v_{i+2}\}$ and $ |B_{i+1}|$ vertices in $B_{i}$, and pick the independent set $B_{i} \setminus V(M_3)$. It means that there are $|B_{i}| +2$ disjoint $K_1$'s and $K_2$'s to cover $T_1 \cup T_2$. Then $\theta_{4} =\theta(H[ T_1 \cup T_2]) \leq |B_{i}| +2 \leq \alpha(H) $ by \eqref{hb=2}. By $\theta_{2} = \theta(H[\{v_{i-2}\} \cup \bigcup_{j=1}^{5} A_j]) \leq 3$, $\theta(H)\leq \theta_{2} + \theta_{4} \leq \alpha(H) +3 $.

\noindent {\bf Case 2.3.} $B_{i} \nsim B_{i+1}$.

Then, by (P7), $B_i = \{b_i\}$ and $B_{i+1} = \{b_{i+1}\}$. Clearly, $V(H) =S_2\cup \{b_i, b_{i+1}\} \cup \bigcup_{j=1}^{5} (\{v_j\} \cup A_j)$. If $S_2 =\emptyset$, then $\theta(H)\leq 6 \leq \alpha(H) +3$ by \eqref{hb=2}. If $S_2 \neq\emptyset$, then $S_2 = \{z\} \sim \{b_i, b_{i+1}\}$ by Lemma \ref{p2-p3 2} (2). Thus $\theta_{5} =\theta(H[\{z, b_i, b_{i+1}\} \cup \bigcup_{i=1}^{5} \{v_i\}]) \leq 4$ and $\theta_{6} =\theta(H[\bigcup_{j=1}^{5} A_j]) \leq 3$. Noting that $\{z, v_i, v_{i+2}, a_{i-2}\}$ is an independent set, $\alpha(H) \geq 3+ |A_{i-2}|$. Furthermore, by \eqref{hb=2}, we have that if $h_A \geq 1$, then $\alpha(H) \geq 4$; if $h_A =0$, then $\alpha(H) \geq 3$ and $\theta_{6}=0$. In either of the cases, $\theta(H)\leq\theta_{5}  + \theta_{6} \leq \alpha(H) +3$.

\bigskip
Thus, in Case 2, $\theta(H)  \leq  2\alpha(H)-2 -\max\{|R_i|, |R_{i+1}|\}\leq 2\alpha(H)-2$ when $H \in \mathcal{B}_1$; $\theta(H)  \leq  \alpha(H)+3$ when $H \notin \mathcal{B}_1$. We next give an example to show that the bound $2\alpha(H)-2$ is tight for graphs in $\mathcal{B}_1$.

\begin{example}\label{b1}
Let $\hat{H}_s\in  \mathcal{H}$ be a graph satisfying that $h_A=0, h_B=2$, $|B_{i}|=| B_{i+1}| = |S_2| =s \geq 2$ for some $i\in [5]$, there are $s$ missing edges (which form a perfect matching) between $B_i$ and $B_{i+1}$, and either of $[S_2, B_i]$, $[S_2, B_{i+1}]$ is a matching of size $s$. (A graph $H_2$ is shown in Figure 3.) Clearly, $\hat{H}_s \in \mathcal{B}_1$ noting that $B_i\setminus N(S_2)=B_{i+1}\setminus N(S_2)=\emptyset$. Since $\hat{H}_s$ is $K_3$-free and $|V(\hat{H}_s)| = 3s+5$, $\lceil(3s+5)/2 \rceil \leq \theta(\hat{H}_s)\leq 2\alpha(\hat{H}_s)-2$. When $s =2$, it is easy to see that $\alpha(\hat{H}_2) = 4$, then $\theta(\hat{H}_2)=6=2\alpha(\hat{H}_2)-2$.
\end{example}

\begin{figure}[ht]
\begin{center}
        \includegraphics[scale=0.2]{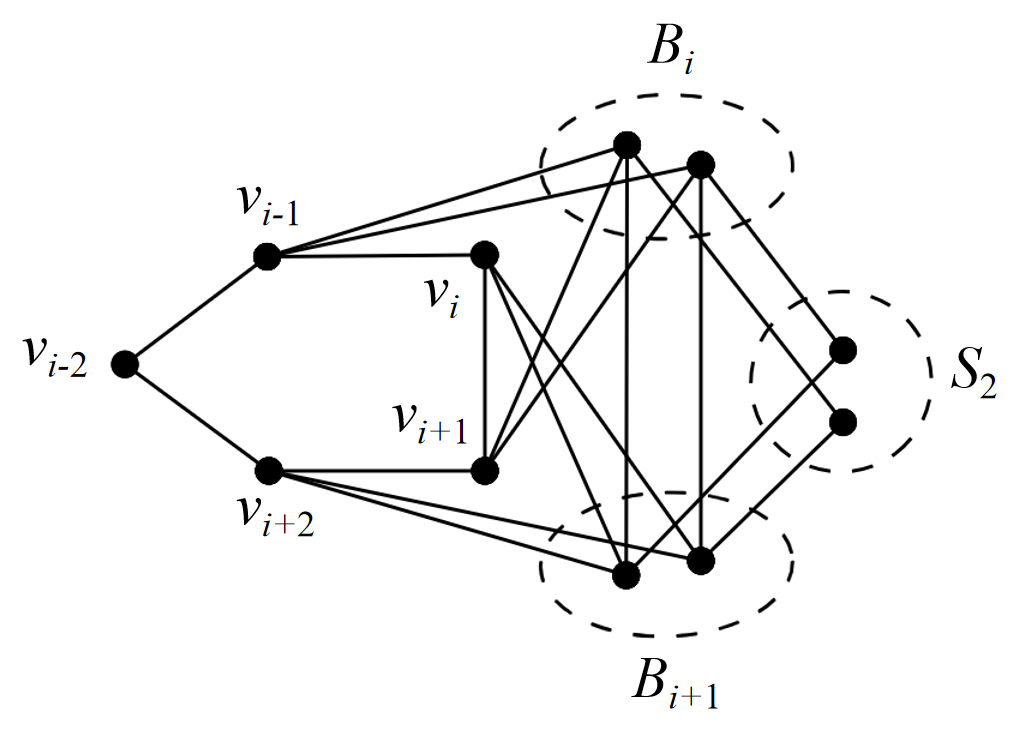}
\caption{Graph $\hat{H}_2.$}
\label{fig3}
\end{center}

\end{figure}

\noindent {\bf Case 3.} $h_B =3$ and $B_{i}, B_{i+1}, B_{i+2} \neq \emptyset$ for some $i \in [5]$.

If $[B_i \cup B_{i+2}, S_2] \neq \emptyset$, then $B_i = \{b_i\}$, $B_{i+2} = \{b_{i+2}\}$ and there exists vertex $z \in S_2$ such that $N(b_i) \cap S_2 = N(b_{i+2}) \cap S_2 = \{z\}$ by Lemma \ref{p2-p3 2} (1). When $A_{i-2} \cup A_{i-1}=\emptyset$, we next show that $\theta(H)\leq \alpha(H) + 3$. By (P10)-(P11), $|S_2 \setminus \{z\}| \leq |B_{i+1}|$ and $[S_2 \setminus \{z\}, B_{i+1}]$ is a matching when $S_2 \setminus \{z\} \neq \emptyset$. Note that
$$X = (S_2\setminus \{z\}) \cup \{v_{i}, v_{i + 2}\},\ \  Y= B_{i+1} \cup \{ v_{i - 2}, v_{i +1}\}$$ are two vertex-disjoint independent sets by (P2) and (P9). Then we may pick a matching $M$ of size $|S_2\setminus \{z\}| +2$ to cover $X$, $\{ v_{i - 2}, v_{i +1}\}$ and $|S_2\setminus \{z\}|$ vertices in $B_{i+1}$, and pick the independent set $Y \setminus V(M)$. It means that there are $|Y|$ disjoint $K_1$'s and $K_2$'s in $M$ and $Y \setminus V(M)$ to cover $X \cup Y$. Then $\theta_{1} =\theta(H[ X\cup Y ]) \leq |Y|$. By (P5), $Y \cup A= B_{i+1} \cup \{ v_{i - 2}, v_{i +1}\} \cup A$ is an independent set for any $A \in \{A_{i}, A_{i+2}\}$, which means that
$$\theta_{1} \leq  |Y| \leq\alpha(H) - |A|.$$
Therefore, when $A_{i-2} \cup A_{i-1}=\emptyset$, $\theta_{2} = \theta(H[V(H) \setminus (X \cup Y )]) =\theta(H[A_i \cup A_{i+1} \cup A_{i+2} \cup \{z, v_{i-1}, b_i, b_{i+2}\}])$. By (P3)-(P5), $H[\{ a_{i+1}, v_{i-1}, b_i, a_{i+2}, a_i, b_{i+2}\}] \cong K_1 \cup P_5$. Note that $z \sim \{b_i, b_{i+2}\}$. Then
 \begin{center}
$\theta_{2} \leq   \left\{
                 \begin{array}{ll}
                   4, & \hbox{when $A_{i}\cup A_{i+2} \neq\emptyset$,}\\
                    3, & \hbox{when $A_{i}\cup A_{i+2} =\emptyset$.}\\
                 \end{array}
               \right.$
\end{center}
Thus $\theta(H) \leq \theta_{1} + \theta_{2} \leq  \alpha(H) + 3$. When $A_{i-2} \cup A_{i-1} \neq\emptyset$, by (P8), we see that $B_{i}\nsim B_{i+1}$ or $B_{i+1} \nsim B_{i+2}$. By Lemma \ref{p2-p3 3}, we have that $ B_{j} =\{b_{j}\}$, $j \in \{i, i+1, i+2\}$, and $S_2 = \{z\} \sim \{b_i, b_{i+1}, b_{i+2}\}$. Noting that $\{b_i, b_{i+2}, v_{i},  v_{i+2}\}$ is an independent set by (P2) and (P6), we have $\alpha(H) \geq 4$. Thus $\theta(H)\leq \theta(H[\{b_i, b_{i+1}, b_{i+2}, z\} \cup \bigcup_{j=1}^{5} \{v_j\} \setminus \{v_{i-1}\}]) + \theta(H[\bigcup_{j=1}^{5} A_j \cup \{v_{i-1}\}]) \leq 4 +3 \leq\alpha(H) + 3$.

If $[B_i \cup B_{i+2}, S_2]= \emptyset$, then by (P2), (P6) and (P9), we may find two vertex-disjoint independent sets
\begin{equation}\label{hb=3, X'-Y}
\begin{aligned}
X'= B_{i} \cup B_{i+2}\cup S_{2} \cup \{v_{i}, v_{i + 2}\}, \ \  Y =B_{i+1} \cup \{ v_{i -2}, v_{i +1}\}.
\end{aligned}
\end{equation}
Let $Q, Q'\in \{B_i \cup B_{i+2}$, $B_{i+1}\setminus N(S_2)\}$ such that
\begin{equation}\label{hb=3, Q-Q'}
\begin{aligned}
|Q| = \max \{|B_i \cup B_{i+2}|, |B_{i+1}\setminus N(S_2)|\}\ \text{and}\  Q' \in  \{B_i \cup B_{i+2}, B_{i+1}\setminus N(S_2)\} \setminus Q.
\end{aligned}
\end{equation}
Then
\begin{equation}\label{hb=3, size}
\begin{aligned}
\alpha(H) \geq \max\{|X'|, |Y|\} =  |S_2| + |Q| +2.
\end{aligned}
\end{equation}
For any vertices $b \in B_i$, $b' \in B_{i+1}$ and $b'' \in B_{i+2}$, we have
\begin{equation}\label{hb=3, 2b=1}
\begin{aligned}
\theta_3 = \theta(H[\{b, b', b''\} \cup \bigcup_{j=1}^{5} v_j \setminus \{v_{i-1}\}]) + \theta(H[\bigcup_{j=1}^{5} A_j \cup \{v_{i-1}\}]) \leq 4 +3 =7.
\end{aligned}
\end{equation}

If $[B_i \cup B_{i+2}, S_2]= \emptyset$, and $B_i \nsim B_{i+1}$ or $B_{i+1} \nsim B_{i+2}$, then $| B_{i}|=| B_{i+1}|=|B_{i+2}|=1$ and $S_2 = \emptyset$ by Lemma \ref{p2-p3 3}. By  \eqref{hb=3, Q-Q'} and \eqref{hb=3, size}, $\alpha(H) \geq 4$. Then, by \eqref{hb=3, 2b=1}, $\theta(H)\leq \theta_3 \leq\alpha(H) + 3$.

If $[B_i \cup B_{i+2}, S_2]= \emptyset$, $[B_i, B_{i+1}] \neq\emptyset$ and $[B_{i+1}, B_{i+2}] \neq\emptyset$, then we discuss two subcases.
One subcase is when $B_{i+1} \setminus N(S_2)= \{b\} \nsim B_i \cup B_{i+2}$. By (P7), we see that $B_i = \{b_i\}$ and $B_{i+2} = \{b_{i+2}\}$. Then $S_2 \neq \emptyset$, for otherwise $B_{i+1} = \{b\} \nsim \{b_i, b_{i+2}\}$, a contradiction to $[B_i, B_{i+1}] \neq\emptyset$ and $[B_{i+1}, B_{i+2}] \neq\emptyset$. So we may pick a matching $M_1 =[S_2, B_{i+1} \cap N(S_2)]$ of size $|S_2|$ by (P11). It follows that $\theta_4 =  \theta(H[S_2 \cup (B_{i+1} \cap N(S_2))])\leq |M_1| = |S_2|$. Furthermore, by \eqref{hb=3, X'-Y} and \eqref{hb=3, 2b=1}, we have $\theta(H)\leq \theta_3 + \theta_4 \leq 7+ |S_2|  = |X'| +3 \leq \alpha(H) + 3$. The remaining subcase is when $B_{i+1} \setminus N(S_2)=\emptyset$, or $|B_{i+1} \setminus N(S_2)| \geq 2$, or $B_{i+1} \setminus N(S_2)= \{b\}$ and $[\{b\}, B_i \cup B_{i+2}] \neq\emptyset$. First, we may pick a matching $M_1$ (being empty when $S_2 = \emptyset$) of size $|S_2|$ to cover all vertices in $S_2 \cup (B_{i+1} \cap N(S_2))$ by (P11), pick two $K_2$'s, $v_{i} v_{i + 1}$ and $v_{i -2} v_{i +2}$, then pick a matching $M_2$ of size $|Q'|$ from $[Q, Q']$ to cover all vertices in $Q'$ and $|Q'|$ vertices in $Q$ by \eqref{hb=3, Q-Q'} and (P7), and finally pick the independent set $Q \setminus V(M_2)$ of size $|Q|- |Q'|$. It means that there are $|M_1| +2+ |M_2| + |Q \setminus V(M_2)|$ disjoint $K_1$'s and $K_2$'s to cover all vertices in $X' \cup Y$. Then $\theta(H[X' \cup Y])\leq |M_1| + 2+ |M_2| + |Q \setminus V(M_2)| = |S_2|  + 2 +|Q|$. Furthermore, by \eqref{hb=3, size}, $\theta(H)\leq \theta(H[X' \cup Y]) + \theta(H[\{v_{i-1}\} \cup \bigcup_{j=1}^{5} A_j]) \leq |S_2|  + 2 +|Q| + 3 \leq \alpha(H) + 3$.

\noindent {\bf Case 4.} $B_{i-1}, B_{i}, B_{i+2}\neq\emptyset$ for some $i \in [5]$.

If $S_2 \neq \emptyset$, or $S_2 = \emptyset$, $h_B \geq 4$ and $B_j \nsim B_{j+1}$ for some $j \in [5]$, or $S_2 = \emptyset$, $h_B \geq 4$ and $h_A \geq 1$, then $|B_j|\leq 1$ for every $j \in [5]$, $\bigcup_{j =1}^{5} B_j$ is an independent set and $S_2= \{z\} \sim \bigcup_{j =1}^{5} B_j$ when $S_2 \neq\emptyset$ by Lemmas \ref{p2-p3 2} (3), \ref{p2-p3 3} and  \ref{p2-p3 5}, respectively. By (P2), (P5) and (P6), $B_{i-1} \cup B_{i+2} \cup \{v_{i-1}, v_{i + 2}\} \cup A_{i-2} $ and $B_{i} \cup B_{i+2} \cup \{v_{i}, v_{i + 2}\} \cup A_{i+1} $ are two independent sets. Then
$$\alpha(H)  \geq \max \{4 + |A_{i-2}|, 4 + |A_{i+1}|, h_B\}. $$
Moreover,
 \begin{center}
$\theta_{1} =\theta(H[V(H)\setminus(\bigcup_{j=1}^{5} A_j \cup \{v_{i}\})]) \leq \left\{
                 \begin{array}{ll}
                   4, & \hbox{when $h_B =3$,} \\
                   5, & \hbox{when $h_B=4, 5$,}
                 \end{array}
               \right.$
\end{center}
and $\theta_{2} =\theta(H[\bigcup_{j=1}^{5} A_j \cup \{v_{i}\}]) \leq 3$. When $h_B =3, 5$, there is $\alpha(H) \geq \theta_{1}$, then $\theta(H)\leq \theta_{1} +\theta_{2} \leq\alpha(H) + 3$. When $h_B =4$, $\theta_{1}  \leq 5$. If $A_{i-2} \cup A_{i+1} \neq\emptyset$, then $\alpha(H) \geq 5\geq\theta_{1}$; if $A_{i-2} \cup A_{i+1}=\emptyset$, then $\alpha(H)  \geq 4\geq\theta_{1}-1$ and $\theta_{2} \leq 2$. So, when $h_B =4$, $\theta(H)\leq \theta_{1} +\theta_{2} \leq\alpha(H) + 3$.

Next we discuss the remaining subcases when $S_2 = \emptyset$. By (P2) and (P6), we may find four independent sets
\begin{equation}\label{case4,xyz}
\begin{aligned}
X&= B_{i-1} \cup B_{i+2} \cup \{v_{i-1}, v_{i + 2}\}, \ \ \ Y= B_i \cup \{v_{i}, v_{i-2}\}, \\
Z&=B_{i} \cup B_{i-2} \cup \{v_{i}, v_{i -2}\}, \ \ \ \ \ \ \ \  W= B_{i} \cup B_{i+2} \cup \{v_{i}, v_{i+2}\}.
\end{aligned}
\end{equation}

When $S_2=\emptyset$ and $h_B =3$, w.l.o.g., let $|B_{i-1}| \geq |B_i|$. Clearly, $V(H) =X \cup Y \cup \bigcup_{j=1}^{5} A_j \cup \{v_{i+1}\}$. If $[B_{i-1}, B_i] \neq\emptyset$, then by (P7), we may pick a matching $M_1$ of size $|B_i| +2$  to cover all vertices in $Y\cup\{v_{i-1}, v_{i + 2}\}$ and $|B_i|$ vertices in $B_{i-1}$, and pick the independent set $X \setminus V(M_1)$. It means that there are $|X|$ disjoint $K_1$'s and $K_2$'s to cover $X \cup Y$. Then $\theta(H)\leq \theta(H[ X \cup Y]) + \theta(H[\bigcup_{j=1}^{5} A_j \cup \{v_{i+1}\}]) \leq |X| + 3\leq \alpha(H) +3$. If $[B_{i-1}, B_i] =\emptyset$, then $B_{i-1}= \{b_{i-1}\}$ and $B_i =\{b_{i}\}$ by (P7). So $V(H) = \{b_{i-1}, b_i\} \cup B_{i+2} \cup \bigcup_{j=1}^{5} (A_j \cup \{v_j\})$. Note that $b_{i-1}v_{i+1}, b_iv_{i-1}, v_{i-2}v_{i+2}, v_{i+1}b \in E(H)$ for any vertex $b \in B_{i+2}$. Combining the independent set $B_{i+2} \setminus \{b\}$, we have $\theta(H[V(H) \setminus\bigcup_{j=1}^{5} A_j])$ $ \leq 4 + |B_{i+2}| -1 = |X|$. Thus $\theta(H)\leq \theta(H[V(H) \setminus\bigcup_{j=1}^{5} A_j]) + \theta(H[\bigcup_{j=1}^{5} A_j]) \leq |X| + 3\leq \alpha(H) +3$.

When $S_2=\emptyset$ and $h_B \geq 4$, for some $i \in [5]$, we assume that $B_{j} \neq \emptyset$ when $j \neq i+1$. According to the above discussion, we may assume that $h_A=0$ and $[B_j, B_{j+1}] \neq \emptyset$, $j \in[5] \setminus  \{i, i+1\}$. Note that it is exactly the subcase that $H \in \mathcal{B}_2$ when $h_B=5$. First, we establish a claim as follows.

\begin{cla}\label{case4,4b}
Let $\theta_3 =\theta(H[ \bigcup_{j=1}^{5} B_j \setminus B_{i+1}])$. Then $\theta_3 \leq  \alpha(H)-2$.
\end{cla}
\begin{proof}
By \eqref{case4,xyz}, w.l.o.g., we may assume that $|Z| \geq |X|$, i.e. $|B_{i} \cup B_{i-2} | \geq |B_{i-1} \cup B_{i+2}|$. Clearly,
\begin{equation}\label{case4,zw}
\begin{aligned}
\alpha(H) \geq \max \{|Z|, |W|\}.
 \end{aligned}
\end{equation}
If $|B_{i-2}| \geq |B_{i+2}|$, then by (P7), $[B_j, B_{j+1}] \neq \emptyset$ $(j \in[5] \setminus  \{i, i+1\})$ and $|B_{i} \cup B_{i-2} | \geq |B_{i-1} \cup B_{i+2}|$, we can pick a matching $M_2$ of size $|B_{i-1} \cup B_{i+2}|$ to cover all vertices in $B_{i-1} \cup B_{i+2}$ and $|B_{i-1} \cup B_{i+2}|$ vertices in $B_{i} \cup B_{i-2}$, and pick the independent set $(B_{i} \cup B_{i-2}) \setminus V(M_2)$ of size $|B_{i} \cup B_{i-2} | - |B_{i-1} \cup B_{i+2}|$. It follows that there are $|B_{i} \cup B_{i-2} |$ disjoint $K_1$'s and $K_2$'s to cover $\bigcup_{j=1}^{5} B_j \setminus B_{i+1}$. By \eqref{case4,xyz} and  \eqref{case4,zw}, we see that
\begin{center}
$\theta_3 = \theta(H[ \bigcup_{j=1}^{5} B_j \setminus B_{i+1}]) \leq |B_{i} \cup B_{i-2} |= |Z|-2\leq  \alpha(H)-2$.
\end{center}
If $|B_{i-2}|< |B_{i+2}|$, then $|B_{i}| >|B_{i-1}|$ by $|B_{i} \cup B_{i-2} | \geq |B_{i-1} \cup B_{i+2}|$. So we can pick a matching $M_3$ of size $|B_{i-2} \cup B_{i-1}|$ to cover all vertices in $B_{i-2} \cup B_{i-1}$ and $|B_{i-2} \cup B_{i-1}|$ vertices in $B_{i} \cup B_{i+2}$, and pick the independent set $(B_{i} \cup B_{i+2}) \setminus V(M_3)$ of size $|B_{i} \cup B_{i+2}| -|B_{i-2} \cup B_{i-1}|$. It means that there are $|B_{i} \cup B_{i+2}|$ disjoint $K_1$'s and $K_2$'s to cover $\bigcup_{j=1}^{5} B_j \setminus B_{i+1}$. By \eqref{case4,xyz} and  \eqref{case4,zw},
\begin{center}
$\theta_3 \leq |B_{i} \cup B_{i+2}| = |W|-2\leq  \alpha(H)-2$.
\end{center}
In either of the situations, $\theta_3 \leq  \alpha(H)-2$.
\end{proof}

When $S_2=\emptyset$ and $h_B =4$, we assume that $B_{i+1} = \emptyset$. As noted above, $h_A =0$. Then $V(H) = \bigcup_{j=1}^{5} (\{v_j\} \cup B_j) \setminus B_{i+1}$. By Claim \ref{case4,4b}, we have $\theta_3 \leq  \alpha(H)-2$. Noting that $\theta(H[\bigcup_{j=1}^{5} \{v_j\} ]) \leq 3$, we have $\theta(H)\leq \theta_3 + \theta(H[\bigcup_{j=1}^{5} \{v_j\} ])  \leq  \alpha(H)-2 + 3 =\alpha(H) + 1$.

When $S_2=\emptyset$ and $h_B = 5$, this is exactly the subcase that  $H \in \mathcal{B}_2$. Here $V(H) = \bigcup_{j=1}^{5} (\{v_j\} \cup B_j)$. W.l.o.g., let $| B_{i+1}|  =\min \{|B_j|, j \in [5]\}$ for some $i \in [5]$, then we have $\alpha(H) \geq \max \{|X|, |Z|\} \geq 2 |B_{i+1}|+2$ by \eqref{case4,xyz}. Then $|B_{i+1}| \leq \alpha(H)/2 -1$, which means $\theta_{4} =\theta(H[B_{i+1} \cup \{v_{i}\}]) \leq \alpha(H)/2 -1$. By Claim \ref{case4,4b} and $\theta_{5} =\theta(H[ \bigcup_{j=1}^{5} \{v_j\} \setminus \{v_{i}\}]) =2$, we have $\theta(H)\leq \theta_{3} + \theta_{4} + \theta_{5}\leq \alpha(H)-2 + \alpha(H)/2 -1 + 2 =  3\alpha(H)/2 -1$.

Totally, in Case 4, $\theta(H) \leq  \alpha(H)+3$ when $H \notin \mathcal{B}_2$; $\theta(H)  \leq  3\alpha(H)/2 -1$ when $H \in \mathcal{B}_2$. We next show that the bound $3\alpha(H)/2 -1$ is tight for graphs in $\mathcal{B}_2$ by virtue of the following claim.

\begin{cla}\label{case4,equa2}
For every $H_t \in \mathcal{B}_2$ with $|B_j| =t \geq 1$ for each $j \in [5]$, $\alpha(H_t) = 2t +2$.
\end{cla}
\begin{proof}
Clearly, $V(H_t) = \bigcup_{j=1}^{5} (\{v_j\} \cup B_j)$ for every $H_t \in \mathcal{B}_2$. When $t=1$, it is easy to see that $\alpha(H_t) =4$. Next we assume that $t \geq 2$.

By \eqref{case4,xyz}, $\alpha(H_t) \geq 2t +2$. To the contrary, suppose that there is an independent set, say $I$, in $H_t$ with size
$|I| \geq 2t +3$.
Then, by $H_t[ \bigcup_{j=1}^{5} \{v_j\}] \cong C_5$, we have
\begin{equation}\label{case4,I}
\begin{aligned}
|I \cap \bigcup_{j=1}^{5} B_j|\geq 2t +1.
 \end{aligned}
\end{equation}
Furthermore, let $D$ denote the set of $\{B_j|B_j \cap I \neq \emptyset, j \in [5]\}$, then we have that $|D| \geq 3$ by $|B_j| =t \geq 2$ for each $j \in [5]$. Then there exists some $i \in[5]$ such that $B_i \cap I \neq \emptyset$ and $B_{i+1} \cap I \neq \emptyset$. By (P7), $|B_i \cap I| = |B_{i+1} \cap I | =1$. So, if $|D| = 3$ and there exactly exist three non-consecutive $B_i$, $B_{i+1}$, $B_{i-2}$  in $D$ for some $i \in [5]$, then $|I \cap \bigcup_{j=1}^{5} B_j|\leq |B_{i-2}| + 2 =t +2<2t +1 $ by $t \geq 2$, a contradiction to \eqref{case4,I}. Moreover, if there exist three consecutive $B_i$, $B_{i+1}$, $B_{i+2}$ in $D$ for some $i \in [5]$, then $|B_j \cap I| = 1$ for each $B_j \in D$ $(j \in [5])$, and $I \cap \bigcup_{j=1}^{5} \{v_j\}  =\emptyset$ by the definition of $B_j$ $(j\in\{i, i+1, i+2\})$. It means that $ 3\leq |D| = |I| \leq 5 <  2t +2$ by $t \geq 2$, a contradiction to  $|I| \geq 2t +3$. Therefore $\alpha(H_t) = 2t +2$.
\end{proof}
\begin{example}\label{b2}
Let $H_t \in  \mathcal{B}_2 \subseteq  \mathcal{H}$ be a graph with $|B_j| =t \geq 1$ for each $j \in [5]$. Then $H_t$ is $K_3$-free, $\theta(H_t)\leq  3\alpha(H_t)/2 -1$ and $|V(H_t)| = 5t+5$. By Claim \ref{case4,equa2}, we have $\lceil (5t+5)/ 2\rceil \leq \theta(H_t) \leq 3\alpha(H_t)/2 -1 =3t+2$. Thus $\theta(H_1)=5=3\alpha(H_1)/2 -1$; $\theta(H_2)=8= 3\alpha(H_2)/2 -1$.
\end{example}
According to the discussion in Cases 1-4, for any $H \in \mathcal{H}$, $\theta(H) \leq \max\{\alpha(H)+3, 2\alpha(H)-2\}$. Particularly, $\theta(H)  \leq 2\alpha(H)-2$ when $H \in \mathcal{B}_1$; $\theta(H)  \leq 3\alpha(H)/2 -1$ when $H \in \mathcal{B}_2$; $\theta(H)  \leq \alpha(H) +3$ when $H \notin \mathcal{B}_1 \cup \mathcal{B}_2$. We have shown that the bounds $2\alpha(H)-2$, $3\alpha(H)/2 -1$ are tight for $\mathcal{B}_1$, $\mathcal{B}_2$ by Example \ref{b1} and Example \ref{b2}, respectively. Here we give another example to show that the bound $\alpha(H)+3$ is tight for graphs in $\mathcal{H} \setminus (\mathcal{B}_1 \cup \mathcal{B}_2)$.

\begin{example}\label{nob1b2}
For graph $H^*$ in Figure $2$, it satisfies that $h_A =5$, $S_2=\{z\}$, $B_{i}=\{b_i\}$ for each $i \in [5]$, $\bigcup_{i=1}^{5} \{b_i\}$ is an independent set and $\{z\} \sim \bigcup_{i=1}^{5} \{b_i\}$. Clearly, $H^* \in \mathcal{H} \setminus (\mathcal{B}_1 \cup \mathcal{B}_2)$. Since $H^*$ is $K_3$-free and  $|V(H^*)| = 16$, there is $|V(H^*)|/2 =8\leq \theta(H^*)\leq  \alpha(H^*) +3$. Moreover, by Theorem \ref{thm-p2p31}, $\alpha(H^*) =(|V(H^*)|-1)/3 =5$. Thus $\theta(H^*)=8= \alpha(H^*)+3$.
\end{example}

We will complete the proof by showing the remaining case that $H$ is an imperfect disconnected graph. Then $H$ is a union of a graph $H' \in \mathcal{H}$ and $n-|H'|$ isolated vertices according to Remark. It follows that $\chi(G)=\theta(H) \leq \theta(H') + n- |H'|  \leq \max\{\alpha(H')+3, 2\alpha(H')-2\} + n- |H'| \leq \max\{\alpha(H) +3, 2\alpha(H)-2\} = \max\{\omega(G) +3, 2\omega(G)-2\}$. In particular, when $H' \in \mathcal{B}_1$, $\theta(H) \leq \theta(H') + n- |H'|  \leq 2\alpha(H') -2 + n- |H'| = 2\alpha(H) -2 -( n- |H'|) < 2\alpha(H) -2$; when $H' \in \mathcal{B}_2$, $\theta(H) \leq \theta(H') + n- |H'|  \leq 3\alpha(H')/2 -1 + n- |H'| = 3\alpha(H)/2 -1 -( n- |H'|)/2 < 3\alpha(H)/2 -1$; when $H' \in \mathcal{H} \setminus(\mathcal{B}_1 \cup \mathcal{B}_2)$, $\theta(H) \leq \theta(H') + n- |H'|  \leq \alpha(H')+3 + n- |H'| = \alpha(H)+3$.

\section{Concluding Remarks}
In 2008, Choudum, Karthick and Shalu \cite{cks} showed that every $\{3K_1, \mathrm{paraglider}\}$-free graph $G$ is $2\omega(G)$-colorable. Theorem \ref{thm-paraglider-boundedness} improves the bound to $\max\{\omega(G)+3, 2\omega(G)-2\}$, which is tight. In particular, the chromatic number is no more than $\omega(G)+3$ for $\{3K_1, \mathrm{paraglider}\}$-free graphs when $G\notin\mathcal{B}_1\cup \mathcal{B}_2$.

In the same paper, Choudum, Karthick and Shalu obtained the following result.

\begin{thm}[\cite{cks}]\label{thm-House, W_4}

Let $G$ be a graph with $\alpha(G) \leq 2$.

\begin{itemize}
\item[\emph {(1)}]  If $G$ is $house$-free, then $\chi(G) \leq \lfloor  3\omega(G)/2 \rfloor$.
\item[\emph {(2)}]  If $G$ is $W_4$-free, then $G$ is a perfect graph, $\mathbb{K}[C_5] + K_t$ for some $t \geq 1$, $\mathbb{K}[\overline{C_{7}}]$, or a generalized buoy $\mathbb{GB}$.
\end{itemize}
\end{thm}

In 2007, Chudnovsky and Ovetsky \cite{co} showed the following result, where a $quasi$-$line$ graph $G$ is one in which every neighborhood $N(v)$ can be expressed as the union of two cliques for each vertex $v \in V(G)$.
\begin{thm}[\cite{co}]\label{thm-quasiline}
If $G$ is a quasi-line graph, then $\chi(G) \leq 3\omega(G)/2$.
\end{thm}

Steiner \cite{ste} in 2023 proved that Odd Hadwiger's Conjecture holds for line-graphs of simple graphs. Note that every line graph is quasi-line, but not vice versa. In \cite{JSWZ}, the authors proposed a problem whether Steiner's result can be extended to quasi-line graphs. By virtue of Theorem \ref{thm-oh-low2}, we may give a positive response for $3K_1$-free graphs. More generally, we can establish the following corollary.

\begin{cor}\label{qhw}
Let $G$ be a graph with $\alpha(G) \leq 2$. If $G$ is quasi-line, $house$-free, or $W_4$-free, then $oh(G) \geq \chi(G)$.
\end{cor}

\begin{proof}
For a quasi-line or house-free graph $G$ with $\alpha(G) \leq 2$, it is clearly true by Theorems \ref{thm-oh-low2}, \ref{thm-House, W_4} (1) and \ref{thm-quasiline}. If $G$ is a $W_4$-free graph with $\alpha(G) \leq 2$, then $\omega(G) \geq n/3$ by Theorem \ref{thm-House, W_4} (2). By Theorem \ref{thm-ohlow} and Lemma \ref{n26}, $oh(G) \geq \chi(G)$.
\end{proof}

Let $\mathbb{H}= \{H: \alpha(H) \leq 2$ and $4 \leq |H| \leq 5\}$. According to Theorem \ref{thm-t1}, Corollary \ref{qhw} and some results in \cite{JSWZ, st}, every $H$-free graph $G$  with $\alpha(G) \leq 2$ has been verified true for Odd Hadwiger's Conjecture for any $H\in \mathbb{H}$ except the following five graphs, $2K_2, K_2 \cup K_3, C_5$, hammer and butterfly.

\end{document}